\documentclass[11pt]{amsart}

\usepackage{amsfonts, amsbsy, amsmath, amssymb}
\usepackage{graphicx,float,hyper ref,listings}

\usepackage{color}
\usepackage{longtable, lscape}

\definecolor{armygreen}{rgb}{0.29, 0.33, 0.13}
\definecolor{amber}{rgb}{1.0, 0.49, 0.0}
\definecolor{amethyst}{rgb}{0.6, 0.4, 0.8}

\newcommand{\cal}{\mathcal}
\newcommand{\GAP}{\textsf{GAP}}
\newcommand{\N}{\mathbb{N}}

\newcommand{\Z}{\mathbb{Z}}

\newcommand{\Sym}{\mathbb{S}}

\newtheorem{theorem}{Theorem}[section]
\newtheorem{corollary}[theorem]{Corollary}
\newtheorem{lemma}[theorem]{Lemma}
\newtheorem{proposition}[theorem]{Proposition}
\newtheorem{definition}[theorem]{Definition}

\newtheorem{example}[theorem]{Example}

\newtheorem{remark}[theorem]{Remark}

\newtheorem{conjecture}[theorem]{Conjecture}
\newtheorem{problem}[theorem]{Problem}
\newtheorem{question}[theorem]{Question}
\newtheorem*{question*}{Question}

\input{epsf.sty}
\usepackage{graphicx}

\renewcommand{\to}{\rightarrow}

\newcommand{\kc}[2]{ {\rm Col}_{#1}{(#2)}}
\newcommand{\skc}[2]{ {\rm SCol}_{#1}{(#2)}}

\begin{document}

\title[Quandle Invariants of Composite Knots and Extensions]{Quandle Coloring and Cocycle Invariants of Composite Knots and Abelian Extensions}
\author[E. Clark, M. Saito, L. Vendramin]{W. Edwin Clark, Masahico Saito, Leandro Vendramin}

\address{
E. Clark and M. Saito:
Department of Mathematics and Statistics, University of South Florida, Tampa, Florida, U.S.A.
}
\email{wclark@mail.usf.edu}
\email{saito@usf.edu}

\address{
L. Vendramin:
Departamento de Mathem\'{a}tica, Facultad de Ciencias Exactas y Naturales,
Universidad de Buenos Aires,
Buenos Aires, Argentina
}
\email{lvendramin@dm.uba.ar}

\date{}

\maketitle

%

\begin{abstract}
Quandle colorings and cocycle invariants are studied for composite knots, and
applied to  chirality and abelian extensions.  The square and granny knots, for
example, can be distinguished by quandle colorings, so that a trefoil and its
mirror can be distinguished by quandle coloring 
of composite knots. We
investigate this and related phenomena.  Quandle cocycle invariants   are
studied  in relation to quandle coloring of 
 the connected sum, and  formulas are given for computing
the cocycle  invariant from the number of colorings of composite knots.  Relations to
corresponding abelian extensions of quandles are studied, and extensions are
examined for the table of small connected quandles, called Rig quandles.
Computer calculations are presented, and summaries of outputs are discussed.
%
%
\end{abstract}

\section*{Introduction}

Sets with certain self-distributive operations called  {\it quandles}
have been studied since  the 1940s 
in various areas with different names. 
The {\it fundamental quandle} of a knot
was defined in a manner similar to the
fundamental group \cite{Joyce,Mat} of a knot, which made quandles an important
tool in knot theory.  The number of homomorphisms from the fundamental
quandle to a fixed finite quandle has an interpretation as colorings
of knot diagrams by quandle elements, and has been widely used as a
knot invariant. Algebraic homology theories for quandles 
were defined \cite{CJKLS,FRS1}, and
 investigated in \cite{LN,Moc,NP1,Nos}.
 Extensions of quandles by cocycles have been studied \cite{AG,CENS,Eis3}, and
invariants derived thereof are applied to various properties of knots and knotted surfaces (see \cite{CKS}
and references therein).

Tables of small quandles have been made previously (e.g., \cite{CKS,Clau,EMR}). 
Computations using \GAP~\cite{Leandro} significantly expanded the list for
connected quandles.  
These
quandles may be found in the \GAP~package Rig \cite{rig}.  
Rig includes all connected quandles of order less than 48. 
We refer to
these quandles as { \it Rig} quandles, and use the notation $Q(n,i)$
for the $i$-th quandle of order $n$ in the list of Rig quandles.
As a matrix $Q(n,i)$
is the transpose of the quandle matrix {\sf SmallQuandle}$(n,i)$ in
\cite{rig}. 
In this paper, however, we focus on Rig quandles of order less than 36.
There are 431 such quandles. 

In \cite{CESY}, 
it was investigated 
 to what extent the number of quandle
colorings of a knot by a finite quandle can distinguish the prime oriented
knots with at most 12 crossings in the knot table at KnotInfo
\cite{KI}.
It is known that quandle colorings do not distinguish $K$ from its reversed mirror, $rm(K)$. 
It is also known \cite{CSS} that the quandle cocycle invariant can distinguish 
a trefoil $3_1$ from its mirror image. Since $3_1$ is reversible, it cannot be distinguished from its mirror by quandle colorings.
However,
 we show here that quandle colorings can be used via connected sums
to distinguish $K$ from $rm(K)$ for many knots (we conjecture for all knots $K$ such that $K \neq rm(K)$).  
In particular, for some reversible knots, we can distinguish $K$ from $m(K)$
using this technique. 
For example, by distinguishing  the square and granny knots by quandle colorings, we
distinguish a trefoil from its mirror image. 
In this paper, we investigate this phenomenon, and other properties and applications of 
quandle invariants under connected sum.
In particular, we relate quandle colorings of composite knots to 
quandle $2$-cocycle invariants. 

We also note that  quandle colorings of the  connected sum can be used to
recover  quandle cocycle invariants in many cases. 
It is well-known that
quandle $2$-cocycles give rise to abelian extensions of quandles, see for
example \cite{CENS}. We investigate the relations among abelian extensions that
result from our computations, and their properties. As a result, several
problems arise naturally.

An important part of this work depends on computer calculations. For that
reason, we developed algorithms and techniques for computing quandle
(co)homology groups and explicit quandle $2$-cocycles, abelian extensions of
quandles, dynamical cocycles and non-abelian extensions, colorings and quandle cocycle
invariants of classical and virtual knots. The algorithms are freely available
in the \GAP~package Rig. Several tables with all these calculations are available
online at the Wiki page of Rig: 
\url{http://github.com/vendramin/rig/wiki}.

The paper is organized as follows. Preliminary material necessary for the paper
follows this section, and it is shown that the number of quandle colorings by
finite quandles can distinguish unknot in Section~\ref{sec-unknot}.  Quandle
colorings of composite knots are studied  in Section~\ref{sec-connectsum}.  In
Section~\ref{sec-rev}, quandle colorings of composite knots are applied to
distinguish knots from their 
reversed mirror images,  
relations to the quandle cocycle
invariant are discussed, and computer calculations are presented.  In
Section~\ref{sec-recover}, a method of computing quandle cocycle invariants
from colorings of composite knots is studied.  Relations to abelian extensions
of quandles are examined in Section~\ref{sec-ext}.  Further considerations
regarding extensions of Rig quandles are presented 
in Section 7. For convenience
of the reader, we collect problems, questions and conjectures posed all over
the text in Section 8.

\section{Preliminaries}\label{sec-prelim}

We briefly review some definitions and examples of quandles. 
More details can be found, for example, in \cite{AG,CKS,FRS1}. 
 
A {\it quandle} $X$ is a 
set with a binary operation $(a, b) \mapsto a * b$
satisfying the following conditions.
\begin{eqnarray*}
& &\mbox{\rm (1) \ }   \mbox{\rm  For any $a \in X$,
$a* a =a$.} \label{axiom1} \\
& & \mbox{\rm (2) \ }\mbox{\rm For any $b,c \in X$, there is a unique $a \in X$ such that 
$ a*b=c$.} \label{axiom2} \\
& &\mbox{\rm (3) \ }  
\mbox{\rm For any $a,b,c \in X$, we have
$ (a*b)*c=(a*c)*(b*c). $} \label{axiom3} 
\end{eqnarray*}
 A {\it quandle homomorphism} between two quandles $X, Y$ is
 a map $f: X \rightarrow Y$ such that $f(a*_X b)=f(a) *_Y f(b) $, where
 $*_X$ and $*_Y$ 
 denote 
 the quandle operations of $X$ and $Y$, respectively.
 A {\it quandle isomorphism} is a bijective quandle homomorphism, and 
 two quandles are {\it isomorphic} if there is a quandle isomorphism 
 between them.

\begin{example}
    {\rm
    Any non-empty set $X$ with the operation $a*b=a$ for any $a,b \in X$ is
    a quandle called a \emph{trivial} quandle.
    }
\end{example}

\begin{example}
    {\rm 
    A conjugacy class $X$ of a group $G$ is a quandle with the quandle operation $a*b=b^{-1}
    a b $.  We call this a \emph{conjugation quandle}. 
    }
\end{example}

\begin{example}
    {\rm 
    Let $X$ and $Y$ be quandles. Then $X\times Y$ is a quandle with
    $(x,y)*(x',y')=(x*_Xx',y*_Yy')$ for all $x,x'\in X$ and $y,y'\in Y$.
    }
\end{example}

\begin{example}[Joyce \cite{Joyce}]
    {\rm 
    A \emph{generalized Alexander quandle} is defined  by  
    a pair $(G, f)$ where 
    $G$ is a  group,  $f\in{\rm Aut}(G)$,
    and the quandle operation is defined by 
    $x*y=f(xy^{-1}) y $. 
    If $G$ is abelian, this is called an \emph{Alexander} (or {\it affine})  quandle. 
    }
\end{example}

\begin{example}
    {\rm 
    A function $\phi: X \times X \rightarrow A$ for an abelian group $A$ is called
    a \emph{quandle $2$-cocycle}  \cite{CJKLS} if it satisfies $$ \phi (x, y)-
    \phi(x,z)+ \phi(x*y, z) - \phi(x*z, y*z)=0$$ 
    for any $x,y,z \in X$ and 
    $\phi(x,x)=0$ for any $x\in X$. For a quandle $2$-cocycle $\phi$, $E=X \times A$ becomes a quandle by
    \[
    (x, a) * (y, b)=(x*y, a+\phi(x,y))
    \]
    for $x, y \in X$, $a,b \in A$, denoted by
    $E(X, A, \phi)$ or simply $E(X, A)$, and it is called an \emph{abelian
    extension} of $X$ by $A$. 
    The set of quandle $2$-cocycles of $X$ with coefficients in $A$ is denoted
    by $Z^2_Q(X,A)$.  Two cocycles $\phi_1$ and $\phi_2$ are
    \emph{cohomologous} if there is a function $\gamma\colon X\to A$ such that
    \[
        \phi_2(x,y)=-\gamma(x)+\phi_1(x,y)+\gamma(x*y)
    \]
    for any $x,y\in X$.  The set of equivalence classes is a group and it is
    denoted by $H^2_Q(X,A)$.  See  \cite{CENS} for more information on abelian
    extensions of quandles and \cite{CJKLS,CJKS,CJKS2} for more on quandle
    cohomology.
    }
\end{example}

\begin{example}
{\rm
 In \cite{AG}, \emph{ extensions by constant $2$-cocycles} 
 were
defined as follows.  For a quandle $X$ and a set $S$,  a \emph{constant quandle
cocycle} is a map 
\[
\beta: X \times X \rightarrow {\rm Sym}(S),
\]
where ${\rm Sym}(S)$ is the symmetric group on $S$, such that  $X \times S$ has
a  quandle structure by $(x, t) * (y, s) = (x*y, \beta_{x,y}(t) )$ for $x,y\in
X$ and $s,t\in S$ (see \cite{AG} for details).  This quandle is denoted by $X
\times_{\beta} S$.  The map $\beta$ satisfies the \emph{ constant cocycle
condition} $\beta_{x*y, z} \beta_{x, y} = \beta_{x*z, y*z } \beta_{x, z}$ for
any $x,y,z \in X$ and the quandle condition $\beta_{x,x}={\rm id}$ for any $x
\in X$.  
Following \cite{AG}, we also call these extensions {\it non-abelian} extensions.
}
\end{example}

  Let $X$ be a quandle.
  The {\it right translation}  ${\cal R}_a:X \rightarrow  X$, by $a \in X$, is defined
by ${\cal R}_a(x) = x*a$ for $x \in X$. 
Then ${\cal R}_a$ is a permutation of $X$ by Axiom (2).
The subgroup of 
 ${\rm Aut}(X)$, the quandle automorphism group, %
 generated by the permutations ${\cal R}_a$, $a \in X$, is 
called the {\it inner automorphism group} of $X$,  and is 
denoted by ${\rm Inn}(X)$. 
A quandle is {\it connected} if ${\rm Inn}(X)$ acts transitively on $X$.
A quandle is {\it homogeneous} if ${\rm Aut}(X)$ acts transitively on $X$. %
A quandle is {\it faithful} if the mapping 
$\varphi: X \rightarrow {\rm Inn}(X)$ defined by $ \varphi(a)={\cal R}_a$ is 
an injection 
from $X$ to ${\rm Inn}(X)$.
We note that abelian as well as non-abelian extensions are not faithful.
 The operation $\bar{*}$ on $X$ defined by $a\ \bar{*}\ b= {\cal R}_b^{-1} (a) $
is a quandle operation, and $(X,  \bar{*}) $ is called the {\it dual} quandle of $(X, *)$.
A quandle $X$  is called a {\it kei} \cite{Taka}, or {\it involutory},  if 
$(x*y)*y=x$ for all  $x, y \in X$.

 A \emph{coloring} of an oriented knot
 diagram by a quandle $X$ is a map ${\cal C}$ 
 from the set of arcs ${\cal A}$ of the diagram to $X$ such that the
 image of the map satisfies the relation depicted in
 Figure~\ref{coloredXing} at each crossing.  More details can be found in
 \cite{CKS,Eis3}, for example.  A coloring that assigns the same
 element of $X$ to all the arcs is called trivial, otherwise
 non-trivial.  The number of colorings of a knot diagram by a finite
 quandle is known to be independent of the choice of  diagram, and
 hence is a knot invariant.
 We denote by $\skc{X}{K}$ and $\kc{X}{K}$ the set and the number of colorings of $K$ by $X$.

%

The fundamental quandle is defined in a manner similar to the
fundamental group \cite{Joyce,Mat}.  A {\it presentation} of a quandle
is defined in a manner similar to groups as well, and a presentation
of the fundamental quandle is obtained from a knot diagram (see, for
example, \cite{FR}), by assigning generators to arcs of a knot
diagram, and relations corresponding to crossings.  The set of
colorings of a knot diagram $K$ by a quandle $X$, then, is in
one-to-one correspondence with the set of quandle homomorphisms from
the fundamental quandle of $K$ to $X$.


In this paper all knots are oriented. 
Let $m: \Sym^3 \rightarrow \Sym^3$ be an orientation reversing homeomorphism of  the
$3$-sphere.
For a knot $K$ contained in $\Sym^3$, $m(K)$ is the mirror image of $K$, and $r(K)$ is  the knot $K$
with its orientation reversed.  We regard $m$ and $r$ as maps on equivalence
classes of knots.  We consider the group $\mathcal{G} = \{ 1,r,m,rm \} $ acting
on the set of all oriented knots.  For each knot $K$ let $\mathcal{G}(K) =\{
K,r(K),m(K),rm(K) \}$ be the orbit of $K$ under the action of $\mathcal{G}$.
 
 For knots $K$ and $K'$, we write $K =
K'$ to denote that there is an orientation preserving homeomorphism of 
 $\Sym^3$ that takes $K$ to $K'$ preserving the orientations of $K$
and $K'$.  By a symmetry  we mean that a knot (type) $K$ remains unchanged
under one of $r$, $m$, $rm$.  As in the definition of {\it symmetry type} in
\cite{KI} we say that  a knot $K$ is
\begin{itemize}
\item { {\it reversible} if the only symmetry it has is $K = r(K)$},
\item { {\it negative amphicheiral} if the only symmetry it has is $K = rm(K)$},
\item {  {\it positive amphicheiral} if the only symmetry it has is $K = m(K)$},
\item {{\it chiral} if it has none of these symmetries},
\item {{\it fully amphicheiral} if $K =r(K)=m(K)= rm(K)$}, i.e. if $K$ has all three symmetries.
\end{itemize}
 The symmetry type of each knot on at most 12 crossings is given at \cite{KI}. Thus each of the $2977$ knots $K$ given there represents 
 $1$, $2$ or $4$ 
 knots depending on the symmetry type. 
 Among the $2977$ knots, there are 
 $1580$ reversible, $47$ negative amphicheiral, $1$ positive amphicheiral,  
$1319$ chiral, and $30$ fully amphicheiral knots.
%

It is known \cite{Joyce,Mat} that 
the fundamental quandles of $K$ and $K'$ are isomorphic if and only if 
$K=K'$ or $K=rm(K')$.

\begin{figure}[htb]
    \begin{center}
   \includegraphics[width=3in]{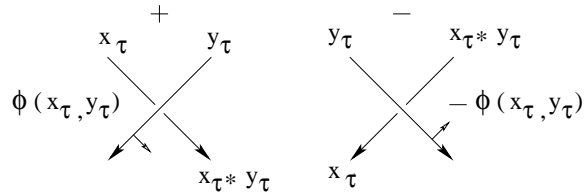}\\
    \caption{Colored crossings and cocycle weights }\label{coloredXing}
    \end{center}
\end{figure}


Let $X$ be a quandle, and $\phi$ be a $2$-cocycle with  coefficient group $A$,
a finite abelian group; we use multiplicative notation. 
We regard $\phi$ as a function $\phi: X \times X \rightarrow A$.
For a coloring  of a knot diagram by a quandle $X$
as depicted in Figure~\ref{coloredXing} at a positive (left) and negative (right) crossing,
respectively, the pair $(x_\tau, y_\tau)$ of colors assigned to a pair of nearby arcs
is called the {\it source} colors. The third arc receives the color $x_\tau * y_\tau$. 

The $2$-cocycle (or cocycle, for short) invariant is an element of the group ring $\Z [A]$ 
defined by $\Phi_{\phi} (K) = \sum_{\cal C} \prod_{\tau} \phi(x_\tau, y_\tau)^{\epsilon(\tau)}$, where
the product ranges over all crossings $\tau$, the sum ranges over all colorings of a 
given knot diagram,
$(x_\tau, y_\tau)$ are source colors at the crossing $\tau$, and $\epsilon(\tau)$ 
is the sign of $\tau$ as specified in Figure~\ref{coloredXing}.


%

When $\Z_n$ is contained as a subgroup in $\Z_m$ and in $H^2_Q(X, \Z_m)$, 
and if a $2$-cocycle $\phi: X \times X \rightarrow \Z_n$ is such  that  $[\phi]$  is a generator of the subgroup 
$\Z_n$ in $H^2_Q(X, \Z_m)$, 
then we say that $\phi$ is a generating $2$-cocycle of the subgroup $\Z_n$. 

\begin{lemma}\label{lem-homcohom} 
If the second homology group $H_2^Q(X,\Z)$ for $X$ satisfies 
$H_2^Q(X,\Z) = \Z_{n_1} \oplus  \Z_{n_2} \oplus  \cdots \oplus  \Z_{n_k}$, $n_i >0$ for all $i$, 
then 
we have
$$H^2_Q(X,\Z_{n} )\cong \Z_{n_1'} \oplus  \Z_{n_2'} \oplus  \cdots  \oplus   \Z_{n_k'} , $$ 
where $n_i'={\rm gcd}(n_{i}, n)$.
\end{lemma}

\begin{proof}
It is known that $H^2_Q(X,A)$ is isomorphic to ${\rm Hom}(H_2^Q(X, \Z), A)$
by the universal coefficient theorem and from the fact that $H_1^Q(X,\Z)$ is torsion free \cite{CJKS2}. 
The result follows from the standard facts
$${\rm Hom}(A_1\oplus  A_2 \oplus  \cdots \oplus  A_k, C) = {\rm Hom}(A_1,C) \oplus  {\rm Hom}(A_2,C) \oplus  \cdots \oplus  {\rm Hom}(A_k,C)$$
and 
${\rm Hom}(\Z_n,\Z_m) \cong \Z_{\gcd(n,m)}$, for positive integers $n$ and $m$.
\end{proof} 

The groups $H_2^Q(X, \Z)$ 
for 
some 
Rig quandles are found at \cite{rig}. Note that the groups given in \cite{rig} are rack homology
$H_2^R(X,\Z)$, and the relationship is given by $H_2^R(X,\Z) \cong H_2^Q(X, \Z) \oplus  \Z$ \cite{LN}.

The package Rig~\cite{rig} includes cohomology groups, $2$-cocycles, abelian
extensions and cocycle invariants for some  Rig quandles and some knots in the
KnotInfo table~\cite{KI}.  Multiplication tables of Rig quandles, (co)homology
groups, generating $2$-cocycles, and abelian extensions of Rig quandles that we
used for computations can be obtained online at the Wiki page of Rig:
\url{http://github.com/vendramin/rig/wiki}.

\section{Distinguishing the unknot by quandle colorings}\label{sec-unknot}

We recall the following conjecture of \cite{CESY}.

\begin{conjecture}\label{conj-separate}
    If $K$ and $K'$ are any two knots such
    that $K' \ne K$ and $K' \ne rm(K)$ then there is a finite quandle $X$ such
    that $ \kc{X}{K} \ne \kc{X}{K'}$.
\end{conjecture}

In this section, we prove  this conjecture when $K'$ is the unknot.
The idea is somewhat similar to that of Eisermann, see
\cite[Remark 59]{Eis3}.

\begin{proposition}\label{prop-unknot}
Let $K$ be a non-trivial knot. Then there exists a finite quandle
$X$ such that $K$ admits a non-trivial coloring with $X$.
\end{proposition}

First we recall the facts we need for the proof, see for
example \cite{Eis3}.

\begin{enumerate}
    \item Papakyriakopoulos~\cite{papa} proved that a knot is trivial if and only
        its longitude is trivial in the fundamental group of the complement of the knot, called the knot group, $\pi_1(\Sym^3 \setminus K)$. 
    \item The {\it Wirtinger presentation} of the knot group of an oriented knot $K$
        is defined as follows. Label the arcs $x_{1},x_{2},\dots, x_n$.
        At the end of the arc $x_{i-1}$ we undercross the arc $x_{k(i)}$
        and continue on arc $x_{i}$. Let $\epsilon(i)$ be the sign of the
        crossing as in Figure~\ref{coloredXing}. Then the knot group is
        \[
        \pi_{1}(\Sym^3 \setminus K)\simeq\langle x_{1},\dots,x_{n}\mid r_{1},\dots, r_{n}\rangle,
        \]
        where $r_{i}={x_{k(i)}}^{-\epsilon(i)}x_{i-1} {x_{k(i)}}^{\epsilon(i)}x_{i}^{-1}$
        for all $i$. 
\item The map 
    $\partial\colon\pi_1(\Sym^3 \setminus K)\to\Z$ given by $\partial(x_i)=1$ for all $i$ is a group homomorphism. 
By \cite{BZ}, 
Remark 3.13, the longitude $l_K$ can be written as a
word $w$ on all the generators $x_1,\dots,x_n$ with $\partial(w)=0$.

\item Recall that a group $G$ is \emph{residually finite} if every non-trivial
$g\in G$ is mapped non-trivially into some finite quotient of $G$.
As a consequence of \cite{Thurston}
one obtains that
every knot group is residually finite, see \cite{Hempel} for a proof.
\end{enumerate}

\begin{proof}[Proof of Proposition~\ref{prop-unknot}]
    Since $K$ is non-trivial, $l_{K}\ne1$. Since knot groups are residually
    finite, there exists a finite group $G$ and a surjective group homomorphism
    $f\colon\pi_{1}(\Sym^3 \setminus K)\to G$ such that $f(l_{K})\ne1$. Then $f$ maps
    the conjugacy class of $x_{1}$ into a non-trivial conjugacy class
    $X$ of $G$. From this it follows that the knot $K$ admits a non-trivial
    coloring with the conjugation quandle $X$.
\end{proof}

%
%
%
%

\section{Quandle colorings of composite knots}\label{sec-connectsum}


In this section we introduce the  concept of {\it end monochromatic}, and show that 
if a knot $K_1$ or 
  a knot   $  K_2$ is  end monochromatic with a finite 
    homogeneous 
    quandle $X$, then 
    $ |X|  \kc{X}{K_1\#K_2}=\kc{X}{K_1} \kc{X}{K_2}$.

A  $1$-tangle is a properly embedded arc in a $3$-ball, and the equivalence of
$1$-tangles is defined by ambient  isotopies of the $3$-ball fixing the
boundary (cf.~\cite{Conway}).  A diagram of a $1$-tangle  is defined in a
manner similar to a knot diagram, from a regular projection to a disk by
specifying  crossing information, see Figure~\ref{tangles}(A).  An orientation
of a $1$-tangle is specified by an arrow on a diagram as depicted.  A knot
diagram is obtained from a $1$-tangle diagram by closing the end points by a
trivial arc outside of a disk. This procedure is called the {\it closure} of a
$1$-tangle.  If a $1$-tangle is oriented, then the closure inherits the
orientation.

\begin{figure}[htb]
    \begin{center}
   \includegraphics[width=2.7in]{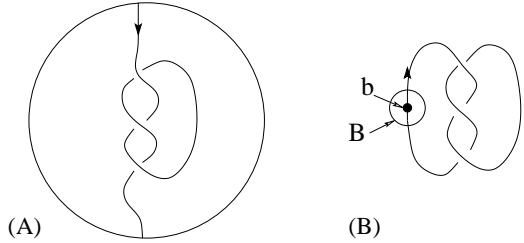}\\
    \caption{ $1$-tangles }\label{tangles}
    \end{center}
\end{figure}

A $1$-tangle is obtained from a knot $K$ as follows.  Choose   a base point $b
\in K$  and a small open neighborhood $B$ of $b$ in the $3$-sphere $\Sym^3$
such that  $(B, K\cap B)$ is a trivial ball-arc pair (so that $K \cap B$ is
unknotted in $B$, see Figure~\ref{tangles}(B)).  Then  $(\Sym^3 \setminus {\rm
Int}(B), K \cap (\Sym^3 \setminus {\rm Int}(B)))$ is a $1$-tangle called the
$1$-tangle associated with $K$.  The resulting $1$-tangle does not depend on
the choice of a base point.  
 If  a knot is oriented, then the corresponding $1$-tangle inherits the
orientation. 

A quandle coloring of an oriented $1$-tangle diagram is defined in a  manner
similar to  those for knots.  We do not require that the end points receive the
same color for a quandle coloring of $1$-tangle diagrams.

\begin{definition}
	{\rm 
	Let  $K$ be a $1$-tangle diagram and $X$ be a quandle.
	We say that $(K, X)$ is \emph{end monochromatic}, 
	or $K$ is \emph{end monochromatic} with $X$, 
	if any coloring of $K$ by $X$ assigns the same color on the two end points.


	}
\end{definition}


%

Two  diagrams of the same $1$-tangle are related by Reidemeister moves.
The one-to-one correspondence of colorings under each Reidemeister move 
does not change the colors of the end points. Thus we have the following.

\begin{lemma}
	The property of being end monochromatic 
	 for a $1$-tangle corresponding to a knot  $K$ 
	and a base point $b$ does not depend on the choice of  the base point $b$.
\end{lemma}

Thus, if a diagram of a $1$-tangle corresponding to 
a knot $K$ and some base point $b$  is end monochromatic with $X$, 
then  we say that a knot $K$ is \emph{end monochromatic} with $X$.

%
%

\begin{lemma}\label{lem-basecolor}
\begin{sloppypar}
    Let $X$ be a finite 
    homogeneous 
    quandle, $x \in X$, 
    and $\kc{ (X, x) }{K,  b}$ 
    be the number of colorings of a diagram $K$ by $X$ such that 
    the arc that contains the base point $b$ receives the color $x$. 
    Then 
    \[
    \kc{ (X, x) }{K,  b} = \kc{X}{K} / |X|
    \]
    for any $x \in X$. 
    \end{sloppypar}
\end{lemma}

\begin{proof}
    First we show that $\kc{ (X, x) }{K,  b} =\kc{ (X, y) }{K,  b} $ for any $x, y \in X$.
    Let  $\skc{ (X, x) }{K,  b}$  be the set of colorings ${\cal C} $ such that  ${\cal C} (\alpha)=x$,
    where $\alpha$ is the arc that contains $b$. 
  Since $X$ is homogeneous, 
   there is an automorphism $h$ of $X$ such that 
    $h(x)=y$. For any coloring ${\cal C} \in \skc{ (X, x) }{K,  b}$, 
    $h_\# ({\cal C} ) =   h \circ {\cal C}$ satisfies $h_\# ({\cal C} ) (\alpha)=y$, hence $h$ induces
    a bijective map
    $h_\# : \skc{ (X, x) }{K,  b} \rightarrow \skc{ (X, y) }{K,  b} $. 
    Then we have 
    $$\kc{X}{K} = \sum_{y \in X} \kc{ (X, y) }{K,  b} = |X| \kc{(X, x)}{K, b}$$
    for any $x \in X$. 
\end{proof}

\begin{figure}[htb]
    \begin{center}
   \includegraphics[width=1.1in]{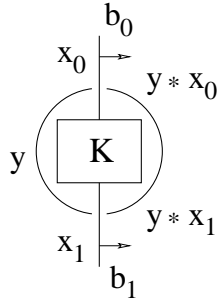}\\
    \caption{ End monochromatic tangle }\label{endmono}
    \end{center}
\end{figure}

The following  lemma was stated and proved in \cite{Jozef} for the $3$ element
dihedral quandle $Q(3,1)$ (and dihedral quandles in \cite{Jozef2004}) %
and generalized by Nosaka~\cite{Nos}.  The idea of
proof is illustrated by Figure~\ref{endmono}, which was taken from
\cite{Jozef}.
 
\begin{lemma}[\cite{Nos}]\label{lem-nosaka1}
    If a quandle $X$ is faithful, then for any knot $K$, $(K, X)$ is end
    monochromatic.
\end{lemma}

\begin{remark}\label{rem-endmonoex}
{\rm
    There are many examples of knots $K$ and quandles $X$ where $X$ is not
    faithful, but $(K, X)$ is end monochromatic.  For example, $Q(8,1)$, which is
    an abelian extension of $Q(4,1)$, is not faithful, but $5_1$ and $8_5$ are end
    monochromatic with $Q(8,1)$, where $5_1$ has only trivial colorings, and $8_5$
    has non-trivial colorings with $Q(8,1)$.  The smallest non-faithful quandle for
    which $3_1$ is end monochromatic is $Q(12, 1)$, which is an abelian extension
    of $Q(6,1)$. 
}
\end{remark}

In the following lemma, a formula is given for the number of colorings of
composite knots.  For a composite knot $K_1 \# K_2$, we assume that $K_1$ and
$K_2$ are oriented, and the composite $K_1 \# K_2$ is defined in such a way
that an orientation of the composite  restricts to the orientation of each
factor, and such an orientation is specified for the composite to make it an
oriented knot, see Figure~\ref{connectsum}.

\begin{lemma}[cf.~\cite{Nos,Jozef}]  \label{lem-nosaka2}
    If a knot $K_1$ or 
a knot    $  K_2$ is 
    end monochromatic  with a finite 
    homogeneous 
    quandle $X$, then 
    $$ |X|  \kc{X}{K_1\#K_2}=\kc{X}{K_1} \kc{X}{K_2}. $$
\end{lemma}

\begin{proof}
    Let $b_1, b_2$ be base points on diagrams of $K_1$ and $K_2$, respectively,
    with respect to which $1$-tangles and connected sum are formed.  Let $x \in
    X$. Let  $\skc{ (X, x) }{K_i, b_i }$, and $\kc{ (X, x) }{K_i, b_i }$,
    $i=1,2$,   be the set and the number of colorings of $K_i$ by $X$ such that
    the arc that contains $b_i$ receives the color $x$.  Let $c_1, c_2$ be
    points on a diagram $K=K_1 \# K_2$ that result from taking a connected sum
    with respect to $b_1$ and $b_2$ by connecting $1$-tangles, see
    Figure~\ref{connectsum}.

    \begin{figure}[htb]
        \begin{center}
            \includegraphics[width=3.5in]{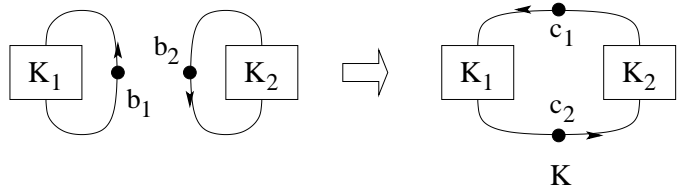}\\
            \caption{ Taking connected sum }\label{connectsum}
        \end{center}
    \end{figure}

    For  colorings ${\cal C}_i \in \kc{ (X, x) }{K_i, b_i }$, $i=1,2$, a coloring
    ${\cal C}= {\cal C}_1 \# {\cal C}_2$ of $K$ is uniquely determined such that
    the colors of the arcs containing $c_i$, $i=1,2$, coincide and is $x$.
Conversely, any coloring ${\cal C}$ of $K$ has the property
    that the color of the arcs containing $c_i$, $i = 1, 2$, coincide, since ${\cal C}$ will also be
    a coloring of the tangles $K_1$ and $K_2$. If, say, $K_1$ is monochromatic with $X$
     then the colors of $c_1$ and $c_2$  must be the same.
Hence there is a bijection $$\bigcup_{x \in X} [\  \skc{ (X, x)
    }{K_1, b_1 } \times \skc{ (X, x) }{K_2, b_2 } \ ]  \rightarrow \skc{ X }{ K }.
    $$ By Lemma~\ref{lem-basecolor}, we have $\kc{(X, x)}{K_i, b_i} = \kc{X}{K_i} /
    |X|$ for any $x \in X$, hence the left side above has the cardinality $$|X| (
    \kc{X}{K_1} / |X| ) (  \kc{X}{K_2} / |X| ) , $$ as desired.
\end{proof}

Lemmas~\ref{lem-nosaka1} and \ref{lem-nosaka2} imply the following.

\begin{lemma}[\cite{Nos}]\label{lem-combined}
    If  $X$ is a finite faithful 
    quandle, then 
    \[
    |X|  \kc{X}{K_1\#K_2}=\kc{X}{K_1} \kc{X}{K_2}
    \]
	for knots $K_1$ and $K_2$.
\end{lemma}

\begin{corollary} \label{cor-rm}
	If $X$ is a finite faithful 
	quandle and $R$, $K$ are knots, then 
	\[
	\kc{X}{R \#K}=\kc{X}{R \# rm(K)}.
	\]
	In particular, if $X$ is a finite faithful quandle and $K$ is reversible or
	positive-amphicheiral, respectively, then either $\kc{X}{R \#K}=\kc{X}{R \#
	m(K)}$ or $\kc{X}{R \#K}=\kc{X}{R \# r(K)}$.
\end{corollary}

\begin{proof}
    By Lemma~\ref{lem-combined}, 
    \begin{eqnarray*}
        \lefteqn{ \kc{X}{R \#K}=\kc{X}{R} \kc{X}{K}/ |X| }\\
        & & = \kc{X}{R} \kc{X}{rm(K)}/ |X| =
        \kc{X}{R \# rm(K)}. 
    \end{eqnarray*}
    This completes the proof.
\end{proof}

According to this lemma, the situation of quandle colorings of composite knots
may differ for non-faithful quandles, and indeed, the computer calculations
reveal this.  In the following sections we investigate these cases.
We used the closed braid form for computer calculations of the number of
quandle colorings as in \cite{CESY}. In computing the number of colorings for
composite knots, we formed the closed braid form as depicted in
Figure~\ref{braidconnect}.  In the braid notation of \cite{KI}, an $m$-braid is
represented by $[a_1,...,a_s]$,  $a_i \in \Z$, where $a_i$ represents the braid
generator $\sigma_k$ if $a_i=k>0$, and $\sigma_k^{-1}$ if $k<0$.  The sign  of
$a_i$, ${\rm sign}(a_i)$,  is defined to be $1$ ($-1$, respectively), if $k>0$
(resp. $k<0$).  If $[a_1,...,a_s]$ ($[b_1,...,b_t]$, respectively) is an
$m$-braid (resp. $n$-braid) representative  for a knot $K$ (resp. $K'$), then
$$ [a_1,\ldots ,a_s,b_1+{\rm sign}(b_1)(m-1), \ldots ,b_t+ {\rm
sign}(b_t)(m-1)] $$ is an $(m+n-1)$-braid representative  for $K\# K'$.  For
example, for a trefoil $3_1$, $s=3$, $m=2$, $t=3$, $n=2$, and $[1,1,1,2,2,2]$
is a $(2+2-1)$-braid representative of $3_1 \# 3_1$.  The orientations of each
factor and the composite are defined by downward orientation of the braid form.
It is known \cite{BiMe} that for the braid index ${\rm Br}$,
the formula ${\rm Br}(K_1 \# K_2) = {\rm Br}(K_1) + {\rm Br}(K_2) -1$ holds. 

\begin{figure}[htb]
    \begin{center}
   \includegraphics[width=3.5in]{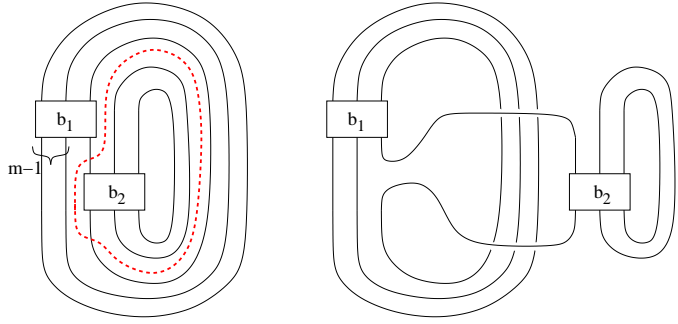}\\
    \caption{The connected sum of two closed braids.}\label{braidconnect}
    \end{center}
\end{figure}

\section{Distinguishing $K$ from $rm(K)$ via colorings of composite knots}\label{sec-rev}

Since quandle colorings do not distinguish $K$ from $rm(K)$, they do not
distinguish $m(K)$ from $r(K)$.  Consequently, in \cite{CESY}, distinguishing
$K$ from $m(K)$ by quandle colorings was examined only for chiral and
negative-amphicheiral knots. 

In this section, we exhibit computational results on distinguishing reversible
and chiral  knots $K$ from $rm(K)$ using quandle colorings of composite knots
$R\# K$ and $R \#rm(K)$ for knots $R$ and $K$. 

\begin{proposition}\label{prop-conj}
	Conjecture~\ref{conj-separate} 
	implies that for any knot $K$ such that $K \neq f(K)$ for some 
	$f \in \mathcal{G}$, there is a finite quandle $X$ and a prime knot $P$ (with braid index $2$) such that 
	$\kc{X}{P \# K} \neq \kc{X}{P \# f(K) } $. 
\end{proposition}

\begin{proof} 
	First we observe that for any knots $K_1$ and $K_2$ and $f \in \mathcal{G}$,
	$$f(K_1 \# K_2)= f(K_1) \# f(K_2),$$ and for any prime knot $P$ and $f \in
	\mathcal{G}$, $f(P)$ is prime.  Let $K=P_1 \# \cdots \#P_n $ be the prime
	factorization of $K$.  Then $$f(K)=f(P_1) \# \cdots \# f(P_n)$$ is the prime
	factorization of $f(K)$.  Let $P$ be a prime knot such that $P$ is not in
	$\mathcal{G}(P_i)$ for $i = 1, \ldots, n$ and $P \ne rm(P)$ (take, for example,
	a $(2, n)$-torus knot, that is, the closure of a $2$-braid, of a large crossing
	number for $P$).  Clearly  $P\#K \neq P\#f(K)$.  The prime factorization of
	$rm(P\#K)$ is $$ rm(P) \# rm(P_1) \# \cdots \#rm(P_n)$$ and by the definition
	of $P$ we again have by uniqueness of prime factorization that $rm(P\#K) $ is
	not equal to $P \#f(K)$. By the conjecture it follows that there is a finite
	quandle $X$ such that $\kc{X}{P \# K} \neq \kc{X}{P \# f(K) } $
\end{proof}

As a corollary to the proof of Proposition~\ref{prop-conj}, we obtain the
following.

\begin{corollary}
    For any knot $K$ such that $K \neq f(K)$ for some $f \in \mathcal{G}$,
    there exists a prime knot $P$ such that the fundamental quandles of $P \#
    K$ and $P \# f(K) $ are not isomorphic.
\end{corollary}

 Recall from
Corollary~\ref{cor-rm} that if  $X$ is a  finite faithful quandle, then we
cannot distinguish $R\# K$ from $R\#  rm(K)$.  Thus to apply this technique, we
must use non-faithful  quandles. 

\begin{remark}\label{rem-cocy}
{\rm
For reversible or chiral prime knots $K$ up to 12 crossings and up to braid
index $4$, among the Rig quandles
of order less than $36$, %
only the quandles $Q(24,2)$ and $Q(27, 14)$
distinguished $R\# K$ and $R \# m(K)$ for some closed $2$-braids $R$ by the
condition $$   \kc{E}{R \# K} \neq  \kc{E}{R \# rm(K)} .$$ We noticed that
these are abelian extensions of $Q(6,2)$ and $Q(9, 6)$ with coefficient groups
$\Z_4$ and $\Z_3$, respectively.  In the remainder of the section, we give an
interpretation of this method in terms of the quandle cocycle invariant, and
extend this method to quandles of order larger than 36. 
Corollary~\ref{cor-rm} and Proposition~\ref{prop-30-4} partly explain why only
abelian extensions worked for this purpose among Rig quandles.
Remark~\ref{rem-endmono} suggests why many abelian extensions do not work. 
}
\end{remark}


Let $X$ be a quandle, $A$ be a finite abelian group, and $\phi \in Z^2_Q(X, A)$
be a $2$-cocycle with coefficient group $A$.  Let $\Phi_\phi(K) = \sum_{g \in A} a_g g \in \Z[A]$ be the
cocycle invariant of a knot $K$.  We write $C_g ( \Phi_\phi(K) ) = a_g$.  In
particular, $C_e( \Phi_\phi(K) ) \in \Z$ denotes  the coefficient of the
identity element $e\in A$. 

An examination of the proof of  Theorem 4.1 in \cite{CENS}
reveals the following two lemmas.
For convenience of the reader, we include a proof of Lemma~\ref{lem-extmono}.

\begin{lemma} [\cite{CENS}] 
	\label{lem-const} 
	Let $E$ be an abelian extension of $X$ with respect to a $2$-cocycle $\phi$
	with coefficient group $A$.  Let $K$ be a knot that is end monochromatic with
	$X$.  
	Then $\kc{E}{K}=C_e( \Phi_\phi(K) ) |A|$. 
\end{lemma}


\begin{lemma} 
\label{lem-extmono} Suppose $(K, X)$ is end monochromatic, and
	$E=E(X, A, \phi)$ is an abelian extension of $X$.  Then $(K, E)$ is end
	monochromatic if and only if  
	$ \Phi_\phi(K) = \kc{X}{K} \, e $. 
\end{lemma}

\begin{proof}
	In \cite{CENS} an interpretation of the cocycle invariant as an obstruction
	to extending a coloring of a knot diagram $K$ by $X$ to a coloring by the
	abelian extension $E$ of $X$ with respect to a $2$-cocycle $\phi$ was
	given as follows. Let ${\cal C}$ be a coloring of a $1$-tangle $S$ of $K$
	with initial and terminal end points $b_0$, $b_1$, respectively.  Suppose
	$(K, X)$ is end monochromatic, so that ${\cal C}(b_0)={\cal C}(b_1)=x_0 \in
	X$.  Let $a_0 \in A$ and assign a color $(x_0, a_0) \in E=X \times A$ to
	the arc at $b_0$.  By traveling along the diagram from $b_0$ to $b_1$, a
	color of $S$ by $E$ is defined inductively using colors by $X$; if an
	under-arc colored by $(x,a)$ goes under an over-arc colored by $(y, b)$ at
	a positive crossing, then the other under-arc receives a color $(x*y, a\,
	\phi(x,y))$.  The color extends at negative crossing as well.  Then the
	coloring thus extended to $S$ has the color $(x_0, a_0\,  d)$ at the arc at
	$b_1$, where $d \in A$ is the contribution of the cocycle invariant $d=
	\prod_{\tau} \phi(x_\tau, y_\tau)^{\epsilon(\tau) } \in A$.  Thus the
	coloring by $X$ extends to that by $E$ if and only if $d$ is the identity
	element.
\end{proof}

\begin{remark}
	\label{rem-endmono}
	{\rm
	The examples mentioned in Remark~\ref{rem-endmonoex} are explained by
	Lemma~\ref{lem-extmono}.  
	Among Rig quandles of order less than 36, the following %
	are abelian extensions
	and end monochromatic for all knots up to 9 crossings:
	\begin{align*}
		& Q(12,1), && Q(20,3), && Q(24,3), && Q(24,4), && Q(24,5),\\
		& Q(24,6), && Q(24,14), && Q(24,16), && Q(24,17), && Q(30,1),\\
		& Q(30,16), && Q(32,5), && Q(32,6), && Q(32,7), && Q(32,8).
	\end{align*}
	Thus we conjecture that this is the case for all knots. The corresponding
	quandle $X$ for these abelian extensions $E$ are found in \cite{CY}, and
	they are, respectively:
	\begin{align*}
		& Q(6,1), && Q(10,1), && Q(12,6), && Q(12,5), && Q(12,8),\\
		& Q(12,9), && Q(12,7), && Q(12,8), && Q(12,8), && Q(15,2),\\
		& Q(15,7), && Q(16,4), && Q(16,4), && Q(16,5), && Q(16,6).
	\end{align*}
	Duplicates in the list of $X$ are due to non-cohomologous $2$-cocycles of
	the same quandle.  

	There are non-faithful quandles that are not abelian extensions, see
	Proposition~\ref{prop-30-4}, and we do not know any characterization  of
	knots that are end monochromatic with such quandles.  All prime knots up to
	9 crossings are end monochromatic with  $Q(30,4)$. 
	}
\end{remark}

\begin{definition}[e.g.~\cite{CJKS}]
	{\rm 
	For an element $a= \sum_h a_h h \in \Z [A]$, the element
	$\overline{a}=\sum_h a_h h^{-1} \in \Z [A]$ is called the {\it conjugate} of
	$a$. 
	}
\end{definition}

\begin{lemma}[\cite{CJKS}] \label{lem-CJKS}
	$\Phi_{\phi} (K)= \overline{ \Phi_{\phi} ( rm(K) ) }$.
\end{lemma}


\begin{definition}
	The value of the quandle cocycle invariant $\Phi_\phi(K)$ of a knot $K$ with
	respect to a $2$-cocycle $\phi$ of a quandle $X$ is called \emph{asymmetric} if
	$\Phi_{\phi} (K) \neq  \overline{ \Phi_{\phi} ( K ) }$. 
\end{definition}

\begin{corollary}\label{cor-phimirror}
	If  $\Phi_\phi(K)$ is asymmetric,
	then $K \neq rm(K)$. 
\end{corollary}


From the above corollary we can sometimes distinguish $K$ from $rm(K)$ using
the cocycle invariant for some quandles.

\begin{proposition}[\cite{Nos}]\label{prop-nosaka}
    Let $\phi$ be a $2$-cocycle of a finite 
    homogeneous 
    quandle $X$ with coefficient group $A$.
    Suppose that $K_1$  or 
    $K_2$ is 
    end monochromatic with $X$.
    Then \[
        |X| \Phi_\phi(K_1 \# K_2) = \Phi_\phi (K_1) \Phi_\phi (K_2).
    \]
\end{proposition}


The following corollary relates the condition $$ \kc{E}{R \# K} \neq  \kc{E}{R
\# rm(K)} $$ to Corollary~\ref{cor-phimirror} via  asymmetry of the cocycle
invariant.

\begin{corollary}\label{cor-RK}
	Let $\phi$ be a $2$-cocycle of a finite connected faithful quandle $X$ with
	coefficient group $A$.  Assume that 
	$\Phi_\phi(R)= r_e e+ r_u u$ for $r_e, r_u \in \N$, the identity
	element $e$, and a non-identity element $u \in A$, and that $r_e=|X|$,
	that is, any non-trivial coloring contribute $u$ to the cocycle invariant.
	Suppose  a knot $K$ satisfies 
	\[
	\Phi_\phi(K)= k_e e+ k_u u + k_{u^{-1}} u^{-1} + V,
	\]
	where $V$ does not contain terms in $e$, $u$ or $u^{-1}$.  Then $k_u \neq
	k_{u^{-1}}$ if and only if 
	\[
	\kc{E}{R \# K} \neq  \kc{E}{R \# rm(K)},
	\]
	where $E$ is  the abelian extension of $X$ by $\phi$. 
\end{corollary}

\begin{proof}
	By Proposition~\ref{prop-nosaka},
	\begin{eqnarray*}
		C_e( \Phi_\phi(R \# K) ) & = & ( r_e k_e + r_u k_{u^{-1}} )/ |X|, \\
		C_e( \Phi_\phi(R \# rm(K) ) ) & = & ( r_e k_e + r_u k_{u} )/|X| .
	\end{eqnarray*}
	By Lemma~\ref{lem-const},   $k_u \neq k_{u^{-1}}$ if and only if 
	\begin{eqnarray*}
		\lefteqn{ \kc{E}{R \# K} = |A| (r_e  k_e + r_u k_{u^{-1}}) /|X| }\\
		& \neq &   |A|( r_e  k_e + r_u k_u) / |X|= \kc{E}{R \# rm(K)} . 
	\end{eqnarray*}
	This completes the proof.
\end{proof}

We note that often  computing the number of colorings  has computational
advantage over applying Corollary~\ref{cor-phimirror} by computing the cocycle
invariant, even though Corollary~\ref{cor-RK} theoretically derives the
condition
\[
\kc{E}{R \# K} \neq  \kc{E}{R \# rm(K)}
\]
from asymmetry of the cocycle invariant in many cases.

\begin{example}\label{ex-69}
	{\rm 
	Let $X=Q(6,2)$ and $\phi$ be a generating $2$-cocycle over $\Z_4$ such that
	the abelian extension of $X$ with respect to $\phi$ is $E=Q(24,2)$.  Let us
	take an example of $R\#K$ and $R\#rm(K)$ for a trefoil $R=3_1$ and $K=6_1$.
	It was found in \cite{CSS} that there is  a multiplicative generator
	$u$ of $\Z_4$ such that the  trefoil has the cocycle invariant $\Phi_{\phi}
	(3_1) = 6+ 24 u$ for $Q(6,2)$. With the same $2$-cocycle, it is computed
	that $\Phi_{\phi} (K) = 6+ 24 u^{-1} $.  By Corollary~\ref{cor-RK},
	$\kc{E}{R \# K} \neq  \kc{E}{R \# rm(K)}$, where $E=Q(24,2)$.  For a more
	complex knot $K$, however, it becomes difficult to compute the cocycle
	invariant, and  easier to confirm  the condition  $\kc{E}{R \# K} \neq
	\kc{E}{R \# rm(K)}$, which then implies  that $K \neq rm(K)$ and  $K$ has
	an asymmetric invariant value.
	}
\end{example}

We summarize outcomes of the methods described in this section, i.e. 
using Corollary~\ref{cor-phimirror} and cocycle invariants, or by directly
computing  
\[
    \kc{E}{R \# K} \neq  \kc{E}{R \# rm(K)}.
\]

First  we summarize our results  for prime knots with 9 crossings or less using the
cocycle invariant.  Among $84$ knots in the table up to 9 crossings, they are
all reversible except: 
\begin{itemize}
	\item
		Fully amphicheiral knots: $4_1, 6_3, 8_3, 8_9, 8_{12}, 8_{18}$.
	\item
		Negative amphicheiral knot:  $8_{17}$.
	\item
		Chiral knots:  $9_{32}, 9_{33}$.
\end{itemize}
The rest are $75$ reversible knots.  The colorings of $ 3_1 \#K$ and $3_1
\#rm(K)$ or the method described in Corollary~\ref{cor-RK} distinguished the
following reversible knots from their mirrors.
\begin{itemize}
	\item
		Using $Q(24, 2)$, the following knots are distinguished from mirrors:\\
		$ 3_1, 6_1, 7_4, 7_7, 8_{11}, $
		$9_1, 9_2, 9_4, 9_6, 9_{10}, 9_{11}, 9_{15}, 9_{17}, 9_{23}, 9_{29},$ \\
		$ 9_{34}, 9_{35},
		9_{37},$ $  9_{38}, 9_{46}, 9_{47}, 9_{48} $.
	\item
		Using $Q(27, 14)$, 
		the following knots are distinguished from mirrors:\\
		$3_1, 6_1, 7_4, 8_5, 8_{15}, 8_{19}, 8_{21},$ 
		$ 9_2, 9_4, 9_{16}, 9_{17}, 9_{28}, 9_{29}, 9_{34}, 9_{38}, 9_{40}$.
\end{itemize}

Furthermore, computer calculations show that the following knots $K$  in the
KnotInfo table up to $12$ crossings with braid index less that $4$  have the
property $\kc{E}{3_1\# K} \neq \kc{E}{3_1\# m(K)}$.
\begin{itemize}
	\item
		Both $E=Q(24,2)$ and $Q(27,14)$ have this property for:\\
		$10_5, 10_9, 10_{112}, 10_{159}, 12a_{0805}, 12a_{0878}, 12a_{1210}, 12a_{1248}$,\\
		$12a_{1283}, 12n_{0571}, 12n_{0666}, 12n_{0750}, 12n_{0751}$. 

	\item 
		Only $E=Q(24,2)$ but not  $Q(27,14)$ has this property for:\\
		$11a_{355}, 12a_{1214}, 12n_{0574}, 12n_{0882}$.

	\item
		Only $E=Q(27,14)$ but not  $Q(24,2)$ has this property for:\\ $10_{64},
		10_{139}, 10_{141}, 11a_{338}, 12a_{1212}, 12n_{0604}, 12n_{0850}$.
\end{itemize}

\begin{sloppypar}
	\begin{remark}
		{\rm
		To distinguish more knots from their mirrors using  the property  $\kc{E}{R \#
		K} \neq \kc{E}{R \# m(K)}$ for some abelian extensions $E$ and for some $R$, we
		further computed abelian extensions of some Rig quandles. We computed
		cohomology groups for some coefficient groups and found some $2$-cocycles for
		Rig quandles up to order 23, and obtained $40$ abelian extensions. This
		information is available online at 
        \url{http://github.com/vendramin/rig/wiki}.

		Let ${\cal E}$ be this set of quandles.  It is likely that there are other
		abelian extensions that are not in this list.

		There are 168  chiral, reversible or positive amphicheiral knots with braid
		index less than  4 and crossing number at most 12.  Of these, we computed that 144
		of 
		knots have the property $\kc{E}{R \# K} \neq \kc{E}{R \# m(K)}$ with
		$E \in {\cal E}$ and for $R=3_1, 5_1,$ or $9_1$.
		}
	\end{remark}
\end{sloppypar}

\begin{remark}
	{\rm
	Reversible prime knots $K$, up to 12 crossings with  braid index less than 4,
	distinguished from their mirror images by a quandle knot pair $(X,R)$ are
	listed in Table \ref{tab:reversible}. The table shows a quandle $X$, a knot $R$
	and knots $K$ such that $\kc{X}{R\#K} \neq \kc{X}{R\#m(K)}$.  We recall  that
	$Q(24,2)$ and $Q(27,14)$ are also abelian extensions.

	\begin{table}[h]
		\caption{Some reversible prime knots $K$ distinguished
		from their mirror images by a quandle knot pair $(X,R)$.}\label{tab:reversible}
		\begin{tabular}{|c|c|c|}
			\hline
			$X$ & $R$ & $K$\\
			\hline
			$E(Q(12,3),\Z_5)$ & $3_1$ & $12n_{0472}$ \\ 
			\hline
			$E(Q(12,3),\Z_5)$ & $9_1$ & $10_{46}, 10_{127}, 10_{155}, 12n_{0466}$\\ 
			\hline
			$E(Q(12,3),\Z_{10})$ & $3_1$ & $8_{7}, 8_{10}, 10_{116}, 10_{143}, 12a_{0576}$\\
			& & $12a_{1220},  12n_{0233}, 12n_{0234}, 12n_{0235}$\\
			& & $12n_{0570}, 12n_{0722},  12n_{0830}, 12n_{0887}$\\
			\hline
			$E(Q(12,3),\Z_{10})$ & $9_1$ & $5_{2}, 10_{2}, 10_{100}, 10_{125}, 10_{152}, 11a_{240}$  \\
			& & $12a_{0835}, 12a_{1203},  12a_{1222},12n_{0242}, 12n_{0467}$\\ 
			\hline
			$E(Q(15,5),\Z_5)$ & $3_1$ & $12n_{0888}$\\ 
			\hline
			$E(Q(18,11),\Z_6)$ & $3_1$ & $8_{20}$, $10_{62}$, $12a_{0999}$, $12n_{0831}$ \\ 
			\hline 
			$E(Q(20,1),\Z_3)$ & $5_1$ & $12a_{1027}, 12a_{1233}, 12n_{0468}, 12n_{0721}$ \\ 
			\hline
			$E(Q(20,1),\Z_6)$ & $5_1$ & $5_{1}, 10_{47}, 10_{48}, 10_{157}, 11a_{234}$ \\
			& & $12a_{0869}, 12a_{1114}, 12a_{1176},  12a_{1199}$ \\ 
			\hline
			$E(Q(20,2),\Z_3)$ & $5_1$ & $7_{3}, 12a_{0146}, 12a_{0369}, 12a_{0722}, 12n_{0822}$ \\ 
			\hline
			$E(Q(20,2),\Z_6)$ & $5_1$ & $6_{2}, 8_{16}, 9_{3}, 10_{126}, 10_{161}, 12a_{0838}, 12a_{1246}, 12a_{1250}$ \\ 
			& & $12n_{0417}, 12n_{0725}, 12n_{0749}, 12n_{0820}, 12n_{0829}$ \\ 
			\hline
			$Q(24,2)$ & $3_1$ & $3_{1}, 9_{1}, 9_{6}, 10_{5}, 10_{9}, 10_{112}, 10_{159}, 11a_{355}$ \\
			& & $12a_{0805},             12a_{0878}, 12a_{1210}, 12a_{1214}, 12a_{1248}, 12a_{1283}$  \\
			& & $12n_{0571}, 12n_{0574}, 12n_{0666}, 12n_{0750}, 12n_{0751}, 12n_{0882}$ \\ 
			\hline
			$Q(27,14)$ & $3_1$ & $8_{5}, 8_{19}, 8_{21}, 9_{16}, 10_{64}, 10_{139}, 10_{141}, 11a_{338}$\\
			& & $12a_{1212},12n_{0604}, 12n_{0850}$ \\
			\hline 
			$E(Q(30,3),\Z_4)$ & $3_1$ & $12n_{0821}$\\
			\hline
		\end{tabular}
	\end{table}
	}
\end{remark}

\begin{remark}
	{\rm
	Chiral prime knots $K$,  up to 12 crossings with braid index less than 4,
	distinguished from $rm(K)$ by a quandle knot pair $(X,R)$ are listed in
	Table \ref{tab:chiral}.  The table shows a quandle $X$, a knot $R$ and
	knots $K$ such that such that $\kc{X}{R\#K} \neq  \kc{X}{R\#rm(K)}$.

	\begin{table}
		\caption{Some chiral prime knots $K$ distinguished
		from $rm(K)$ by a quandle knot pair $(X,R)$.} \label{tab:chiral}
		\begin{tabular}{|c|c|c|}
			\hline
			$X$ & $R$ & $K$\\
			\hline
			$E(Q(12,3),\Z_5)$ & $3_1$ & $10_{149}, 12n_{0344}, 12n_{0679}, 12n_{0688}$ \\ 
			\hline 
			$E(Q(12,3),\Z_{10})$ & $3_1$ & $12a_{0815}, 12a_{0898}, 12a_{0981}, 12n_{0708}$\\ 
			\hline 
			$E(Q(12,3),\Z_{10})$ & $9_1$ & $12a_{1223}, 12n_{0748}$\\ 
			\hline
			$E(Q(18,11),\Z_6)$ & $3_1$ & $10_{82}, 12a_{1191}, 12a_{1215}, 12a_{1253}, 12n_{0675}$ \\ 
			\hline
			$E(Q(20,1),\Z_3)$ & $5_1$ & $10_{148}, 12a_{1047}, 12a_{1227}$ \\ 
			\hline
			$E(Q(20,1),\Z_6)$ & $5_1$ & $12a_{0824}, 12a_{0850}, 12a_{0859}, 12n_{0113}, 12n_{0114}, 12n_{0345}$ \\ 
			\hline 
			$E(Q(20,2),\Z_3)$ & $5_1$ & $12a_{1227}, 12a_{1235}, 12a_{1258}$ \\ 
			\hline
			$E(Q(20,2),\Z_6)$ & $5_1$ & $12a_{0920}, 12n_{0709}$ \\ 
			\hline
			$Q(24,2)$ & $3_1$ & $10_{106}, 12a_{0909}, 12a_{0916}, 12a_{1002}, 12a_{1120}, 12a_{1226}$ \\ 
			& & $12a_{1255}, 12n_{0640}, 12n_{0767}$ \\ 
			\hline
			$Q(27,14)$ & $3_1$ & $10_{85}, 12a_{0864}, 12a_{1219}, 12a_{1221}, 12n_{0674}$ \\
			\hline
			$E(Q(30,3),\Z_4)$ & $3_1$ & $12a_{1011}, 12a_{1051}, 12n_{0191}, 12n_{0684}$\\
			\hline
		\end{tabular}
	\end{table}
	}
\end{remark}

\section{Recovering cocycle invariants from colorings}\label{sec-recover}

In this section, we obtain formulas for computing the cocycle invariant from
the number of colorings
for some cases.  The formulas give computational advantage in many
cases. To obtain formulas, however, one needs information on concrete
non-trivial invariant values for a few knots.

\begin{proposition}\label{prop-invcoeff}
	Let $X$, $A$, $\phi$ be as above.  Suppose that $X$ is end monochromatic
	with $K$.  Suppose further that for an element $v \in A$ that is not the
	identity element $e$, there exists a knot $R_v$ such that $\Phi_\phi(R_v)=
	r_e e + r_v v \in \Z[A]$.  Then $$ C_{ v^{-1} } ( \Phi_\phi(K) )=
	\frac{1}{r_v |A|} ( \  |X|  \kc{E}{ R_v \# K } - r_e\,  \kc{E}{ K } \ ) .
	$$
\end{proposition}

\begin{proof}
	By Proposition~\ref{prop-nosaka}, we have $|X| \Phi_\phi(R_v \# K) =
	\Phi_\phi (R_v) \Phi_\phi (K)$.  By assumption $ \Phi_\phi (R_v) \Phi_\phi
	(K) = (r_e e + r_v v  ) (\sum_{u \in A} a_u u )$.  The coefficient of the
	identity element in the left-hand side is $r_e a_e + r_v a_{v^{-1}}$.
	Hence we obtain $|X| C_e( \Phi_\phi(R_v \# K) ) = r_e a_e + r_v a_{v^{-1}}
	$.  Let $E$ be the abelian extension of $X$ with respect to $\phi$.  Then
	by Lemma~\ref{lem-const}, we have
    \[
        \kc{E}{ R_v \# K) } = C_e( \Phi_\phi(R_v \# K) ) |A|
    \]
	and $\kc{E}{ K }= a_e |A|$.  By substitution and solving for $a_{v^{-1}}$,
	we obtain the lemma.
\end{proof}

In the following examples, we focus on the Rig quandles of order up to 12 where
the second cohomology group is non-trivial when the coefficient group is  other
than  $\Z_2$.  When the coefficient group $A$ is cyclic of order $n$, even
though we write $A=\Z_n$ (a notation usually used for the additive group of
integers modulo $n$), we specify a multiplicative generator $u$, so that
$A=\langle u \rangle$ where $u$ has order $n$, and write $A$ multiplicatively.

\begin{example}\label{ex-62}
	{\rm
	Let $X=Q(6,2)$ and $\phi$ be a generating $2$-cocycle over $A=\Z_4$ such that
	the abelian extension of $X$ with respect to $\phi$ is $E=Q(24,2)$.  Since $X$
	is faithful, any knot is end monochromatic with $X$. 

	The cocycle invariants of  $X=Q(6,2)$ using this cocycle are given in
	the wiki page of Rig at 
	\url{http://github.com/vendramin/rig/wiki}, 
	for knots up to
	10 crossings.  Some of the results are shown in Table~\ref{tab:Q(6,2)}.  
	Knots that are not listed have the trivial invariant value $6$.
	We abbreviate the identity element in the remaining of the paper. For example,
	$6+24u$ means $6e + 24u$ for the identity element $e$.
	In particular, in order to use Proposition~\ref{prop-invcoeff}, we obtain the following invariant values:
	\begin{align*}
		\Phi_{\phi} (3_1) = 6+ 24 u, &&
		\Phi_\phi(8_5)=30 + 24 u^2, &&  
		\Phi_\phi(9_1)=6 + 24 u^3.
	\end{align*}
	Proposition~\ref{prop-invcoeff} implies that
	\begin{eqnarray*}
		C_{ u } ( \Phi_\phi(K) )&=& (\ 1/ (24 \cdot 4 ) \ ) \ ( \  6 \cdot  \kc{E}{  9_1 \# K } - 6 \cdot \kc{E}{ K } \ ) , \\
		C_{ u^2 } ( \Phi_\phi(K) )&=& (\ 1/ (24 \cdot 4 ) \ ) \ ( \  6 \cdot    \kc{E}{  8_5 \# K } -  30 \cdot  \kc{E}{ K } \ ),  \\
		C_{ u^3 } ( \Phi_\phi(K) )&=& (\ 1/ (24 \cdot 4 ) \ ) \ ( \  6 \cdot  \kc{E}{  3_1 \# K } - 6 \cdot \kc{E}{ K } \ ) .
	\end{eqnarray*}
	We also have 
	$$C_{ e } ( \Phi_\phi(K) ) = (1/|A|)  \kc{E}{K}=(1/4) \kc{E}{K}$$
	from Lemma~\ref{lem-const}.
	Therefore we obtain 
	\begin{eqnarray*}
		\lefteqn{\Phi_\phi(K) =
		\frac{1}{16} \left[ \right.
		4\,  \kc{E}{ K } + (   \kc{E}{  9_1 \# K } - \kc{E}{ K } ) u } \\
		& & \quad \quad  +\  ( \kc{E}{  8_5 \# K } -  5\, \kc{E}{ K })u^2 +  ( \kc{E}{  3_1 \# K } -  \kc{E}{ K }) u^3 \
		\left. \right].
	\end{eqnarray*}
	See the appendix for examples of cocycles invariants computed using this formula.

	\begin{table}
		\caption{Some cocycle invariants for the quandle $Q(6,2)$.}\label{tab:Q(6,2)}
		\begin{tabular}{|c|c|}
			\hline 
			cocycle invariant & knot\\
			\hline
			$6 + 24 u$ &  $3_1, 7_7, 8_{11}, 9_2, 9_4, 9_{10},  9_{11}, 9_{15}$  \\
			\hline
			$54 + 72 u$ & $9_{35}$ \\
			\hline
			$6 + 48 u^2$ &  $9_{40}$ \\
			\hline
			$30 + 24 u^2$  &    $8_5, 8_{10}, 8_{15},  8_{19}, 8_{20}, 8_{21},  9_{16},  9_{24}, 9_{28}$ \\
			\hline
			$6 + 24 u^3$ &    $6_1, 7_4, 9_1, 9_6, 9_{17}, 9_{23},  9_{29},  9_{34}, 9_{38}$   \\
			\hline
			$6 + 72 u + 48 u^2$ &  $9_{48}$ \\
			\hline
			$6+ 48 u + 48 u^3$ &  $8_{18}$ \\
			\hline
			$6+ 24u + 72 u^3$ &  $9_{47}$ \\
			\hline
			$54 + 24 u + 48 u^3$ &  $9_{46}$ \\
			\hline
			$6+ 48 u + 48 u^2 + 24 u^3$ &  $9_{37}$\\
			\hline
		\end{tabular}
	\end{table}
	}
\end{example}

\begin{remark}
{\rm 
    In computing the coloring numbers of knots by quandles, some computational
    techniques have been developed in \cite{CESY}, such as fixing a color of the
    first braid strand to reduce the computation time. On the other hand, to
    compute the cocycle invariant, every coloring must be computed, and the cocycle
    value must be evaluated for each coloring.  The latter increases the
    computational time significantly.  Thus the formula of
    Proposition~\ref{prop-invcoeff} is useful in determining invariant values for
    higher crossing knots with  lower braid indices.  
}
\end{remark}

\begin{remark}
	{\rm
	There are discrepancies of representatives of knots and their mirrors in
	different notations in \cite{KI} for the following knots up to 9 crossings:
	$7_7$, $9_{11}$, $9_{17}$,  $9_{34}$, $9_{46}$, $9_{47}$, $9_{48}$.
	Specifically, the diagram of $7_7$ listed agrees with  the braid notation,
	but its PD notation seems to represent its mirror.  In our first
	computation up to 9 crossings, we used the PD notation in \cite{KI}, and
	the second computations for those with braid index less than 4 are
	performed using the braid notation.  For up to 9 crossings, these
	calculations showed discrepancies for the above listed knots.  The
	discrepancies are all related by conjugate values of the invariant.  We
	note that in the following computations, these knots are not used for $R$
	in $ \kc{E}{  R \# K } $ in the formulas.
	}
\end{remark}

Below we give a summary of the formula in Proposition~\ref{prop-invcoeff} for
Rig quandles of order up to 12, as examples  to indicate how to use the
formula, and to illustrate varieties of actual formulas obtained.

\begin{example}
	{\rm
	Let $X=Q(9,6)= \Z_3 [t] / (t^2 + 2t +1) $ and $\phi$ be a generating
	$2$-cocycle over $A=\Z_3$ such that the abelian extension of $X$ with
	respect to $\phi$ is $E=Q(27,14)$.  Since $X$ is faithful, any knot is end
	monochromatic with $X$.  Computer calculation shows that $\Phi_{\phi} (3_1)
	= 27 + 54 u$, where $u$ is a multiplicative generator of $A$ and it also
	implies that $\Phi_{\phi} (m(3_1)) = 27 + 54 u^2$.
	Proposition~\ref{prop-invcoeff} implies that
	\begin{eqnarray*}
		C_{ u } ( \Phi_\phi(K) )&=& (\ 1/ (54 \cdot 3 ) \ ) \ ( \  9 \cdot  \kc{E}{  m(3_1) \# K } - 27 \cdot \kc{E}{ K } \ ) , \\
		&=& (\ 1/ 18  \ ) \ ( \   \kc{E}{  m(3_1) \# K } - 3 \cdot \kc{E}{ K } \ ) , \\
		C_{ u^2 } ( \Phi_\phi(K) )&=& (\ 1/ 18\ ) \ ( \     \kc{E}{  3_1 \# K } -  9 \cdot  \kc{E}{ K } \ ).
	\end{eqnarray*}
	}
\end{example}

\begin{example}
	{\rm
	Let $X=Q(12,3)$. This quandle is  not Alexander, not kei, not Latin,
	faithful, and $H^2_Q(X, A)=\Z_{10}$ for $A=\Z_{10}$.  Let $E$ be the
	abelian extension corresponding to a cocycle that represents a generator of
	$\Z_{10}$.  We obtain the following invariant values:
	\begin{align*}
		&\Phi_\phi( 3_1)= 12 + 60 u, && 
		\Phi_\phi( 8_{19} )=12 + 60 u^2 ,&&
		\Phi_\phi( 5_2)= 12 + 60 u^3,&&\\
		&\Phi_\phi( m(9_{29}))= 12 + 60 u^6, &&
		\Phi_\phi( 5_1)= 12 + 60 u^5,&&
		\Phi_\phi( 9_{29})= 12 + 60 u^6,&&\\
		&\Phi_\phi( 5_2)= 12 + 60 u^7,&&
		\Phi_\phi( m(8_{19}) )=12 + 60 u^8 ,&&
		\Phi_\phi( 8_6)= 12 + 60 u^9.
	\end{align*}
	One computes
	\begin{eqnarray*}
		C_{ u} ( \Phi_\phi(K) ) &=&  (\ 1/ (60 \cdot 12 ) \ ) \ ( \  12 \cdot  \kc{E}{  8_6 \# K } - 12 \cdot \kc{E}{ K } \ ) \\
		&=& ( 1/ 60    ) \ ( \  \kc{E}{  8_6 \# K) } -  \kc{E}{ K } \ ) 
	\end{eqnarray*}
	and the other terms are similar with the corresponding knots listed above.
	We note that the coefficient of every term is computed by these formulas,
	but we needed to compute the invariant for up to 9 crossings for this
	conclusion, as $u^4$ and $u^6$ are missing up to 8 crossing knots.
	}
\end{example}

\begin{example}
	\label{exa:Q(12,5)}
	{\rm
	Let $X=Q(12,5)$. This quandle is not Alexander, not kei, not Latin,  faithful,
	and $H^2_Q(X,\Z_4)=\Z_4$. 
	With a choice of a generating cocycle $\phi$, 
	up to $8$ crossings, all knots have the cocycle invariant 
	of the form $\Phi_\phi(K)=a+b u^2$, $a, b \in \Z$.
	Thus we conjecture that this is the case for all knots.
	The trefoil has the invariant value 
	$\Phi_\phi(3_1)=12 + 96 u^2$.
	Hence we obtain
	\begin{eqnarray*}
		C_{ u^2 } ( \Phi_\phi(K) ) &=&  (\ 1/ (96 \cdot 4 ) \ ) \ ( \  12 \cdot    \kc{E}{  3_1 \# K } -  12 \cdot  \kc{E}{ K } \ )\\
		&=&(1/ 32 ) \ ( \   \kc{E}{  3_1 \# K } -   \kc{E}{ K } \ ) . 
	\end{eqnarray*}
	If the conjecture does not hold and a knot with the term $u$ or $u^3$ is found, then 
	it can be used to evaluate other terms.
	}
\end{example}

\begin{example} \label{exa:Q(12,6)}
	{\rm
	Let $X=Q(12,6)$. This quandle is not Alexander, not kei, not Latin,
	faithful, and $H^2_Q(X,\Z_4)=\Z_4$.  With a generating $2$-cocycle $\phi$
	of $\Z_4$ the invariant values $\Phi_\phi(K)$ for $K$ up to 9 crossing
	knots are listed in Table \ref{tab:Q(12,6)}. 
	\begin{table}
		\caption{Some cocycle invariants for $Q(12,6)$.}
		\begin{tabular}{|c|c|}
			\hline
			cocycle invariant & knot\\
			\hline
			$108$ & $3_1, 6_1, 7_4, 7_7, 8_{11}$ \\
			& $9_1, 9_2, 9_4, 9_6, 9_{10}, 9_{11}, 9_{15}, 9_{17}, 9_{23}, 9_{29}, 9_{34}, 9_{38}$ \\
			\hline
			$204$ & $8_5, 8_{10}, 8_{15}, 8_{19}, 8_{20}, 8_{21}, 9_{16}, 9_{24}, 9_{28}$ \\
			\hline
			$396$ & $8_{18}, 9_{40}, 9_{47}$ \\
			\hline
			$492 + 192 u^2$ & $9_{35}, 9_{37}, 9_{46}, 9_{48}$\\ 
			\hline
			$12$ & otherwise \\
			\hline
		\end{tabular}
		\label{tab:Q(12,6)}
	\end{table}
	Thus we conjecture that the invariant values are of the form 
	\[
	\Phi_\phi(K)=a+b u^2, 
	\]
	for $a,b \in \Z$ and for all knots $K$. One computes 
	\begin{eqnarray*}
		C_{ u^2 } ( \Phi_\phi(K) ) &=&  (\ 1/ ( 492 \cdot 4 ) \ ) \ ( \  12 \cdot    \kc{E}{  9_{35} \# K } -  12 \cdot  \kc{E}{ K } \ )\\
		&=&(1/ 164 ) \ ( \   \kc{E}{  9_{35} \# K } -   41 \cdot \kc{E}{ K } \ ) .
	\end{eqnarray*}
	We note that we needed to compute the invariant for knots up to 9 crossing
	to obtain this formula. 
	}
\end{example}

\begin{remark}
	{\rm
	The second cohomology groups for $Q(12, 7)$, $Q(12, 9)$ with coefficient group $\Z_4$
	are
	$\Z_2 \times \Z_4$ and $\Z_4 \times \Z_4$, respectively, and for choices of generating cocycles,
	the cocycle invariants are non-trivial. 
	Situations and computations  are similar to those for $Q(6,2)$ and $Q(12, 5)$
	for each factor, for up to 7 crossings.
	}
\end{remark} 

\begin{example}
	\label{exa:Q(12,10)} 
	{\rm
	Let $X=Q(12,10)$. This quandle  is not Alexander, not kei, not Latin,
	faithful, and $H^2_Q(X,\Z_6)=\Z_6$.  With a generating cocycle $\phi$ of $\Z_6$, we
	obtain 
	\begin{align*}
		\Phi_\phi(3_1)=12 + 108 u^3, &&
		\Phi_\phi(8_5)=120 + 216 u^2,&& 
		\Phi_\phi(8_{15})=120 + 216 u^4.
	\end{align*}

	Since we observed, up to 8 crossings,  one or more of the terms with $u^2, u^3$ and $u^4$
	(and no terms of $u$ or  $u^5$), 
	we conjecture that it is the case for all knots.
	One computes
	\begin{eqnarray*}
		C_{ u^2} ( \Phi_\phi(K) )&=& (\ 1/ (108 \cdot 6 ) \ ) \ ( \  12 \cdot  \kc{E}{  3_1 \# K) } - 12 \cdot \kc{E}{ K } \ ) , \\
		&=& (\ 1/ 54 \ ) \ ( \  \kc{E}{  3_1 \# K) } - \kc{E}{ K } \ ) ,\\
		C_{ u^3 } ( \Phi_\phi(K) )&=& (\ 1/ (216 \cdot 6 ) \ ) \ ( \  12 \cdot    \kc{E}{  8_{15}  \# K) } -  120 \cdot  \kc{E}{ K } \ ),  \\
		&=& 
		(\ 1/ 108 \ ) \ ( \    \kc{E}{  8_{15}  \# K) } -  12 \cdot  \kc{E}{ K } \ ),  \\
		C_{ u^4 } ( \Phi_\phi(K) )&=& (\ 1/ 108 \ ) \ ( \   \kc{E}{  8_5 \# K) } - 12 \cdot \kc{E}{ K } \ ) .
	\end{eqnarray*}
	}
\end{example}

\begin{remark}
	{\rm
	The $2$-cocycle invariants discussed in this section are derived from the
	following invariant: Let $R_1, \ldots, R_n$ be knots and $X_1, \ldots, X_m$
	be finite quandles.  Then an invariant is defined for a knot $K$ by $${\rm
	CL}_{X_1, \ldots , X_m, R_1, \ldots, R_n}(K) = \left[ \  \kc{X_i}{R_j\#K} \
	\right]_{i=1, \ldots, m,\,  j=1, \ldots, n} .$$ It is, then, a natural
	question whether for any quandle $2$-cocycle invariant $\Phi_\phi(K)$,
	there is a sequence of knots $R_1, \ldots, R_n$  and quandles $X_1, \ldots,
	X_m$ such that $\Phi_\phi(K)$ is derived from ${\rm CL}_{X_1, \ldots , X_m,
	R_1, \ldots, R_n}(K) $.
	}
\end{remark}

\section{Properties of abelian extensions}\label{sec-ext}

Finding abelian extensions have, for example, the following applications:
(1) non-triviality of the second cohomology group can be confirmed,
(2) knots and their mirrors may be distinguished by colorings of composite knots as in Section~\ref{sec-rev}, 
(3) they are useful in computing cocycle knot invariants  via colorings as in Section~\ref{sec-recover}.

We summarize our findings on extensions of Rig quandles in this section. Among
the $790$ Rig quandles of order $<48$ there are 66 non-faithful quandles. All but $8$ 
are extensions by $\Z_2$.

\begin{proposition}\label{prop-30-4}
    Among the non-faithful  Rig quandles (of order less than 48), 
    $Q(30, 4)$, $Q(36, 58)$, and $Q(45, 29)$ are  the only quandles that are not  abelian extensions. 
\end{proposition}

\begin{proof}
	Computations show that the only non-trivial quotient of $Q(30,4)$ is
	$X=Q(10,1)$.  So it suffices to show that there is no abelian extension of
	$X$ 
	of order 30.  We have $H_2^Q(X, \Z)\cong \Z_2$ \cite{rig}.  To get an
	abelian extension of $X $ of order 30 we would have to have a
	non-trivial 2-cocycle $X\times  X \rightarrow \Z_3$ which  would give an
	element of $H^2_Q(X,\Z_3) ={\rm Hom}(\Z_2,\Z_3) = 0$, a contradiction.
	
	The only non-trivial quotients of $Q(36,58)$ are $Q(4,1)$ and $Q(12,10)$.
	Since $H_2(Q(4,1))\cong \Z_2$,  a similar argument implies 
	that $Q(36,58)$ is not an abelian extension of $Q(4,1)$.
	We have $H_2(Q(12, 10)) \cong \Z_6$, and let $f$ be a $2$-cocycle 
	that generates $H^2(Q(12, 10), \Z_6)\cong \Z_6$. Then  
	$2f$ and $4f$ take values in $\Z_3$, and computations show that the corresponding abelian extensions are both isomorphic to $Q(36, 57)$. Since cohomologous cocycles give rise to isomorphic quandles,
	this implies that $Q(36,58)$ is not an abelian extension of $Q(12, 10)$. 
	
	The only  non-trivial quotient of $Q(45, 29)$ is $Q(15,7)$. 
	Since $H_2(Q(15,7))\cong\Z_2$, a similar argument implies that 
	$Q(45, 29)$ is not an abelian extension of $Q(15,7)$.
	
	Then one checks by computer that all the other non-faithful Rig quandles
	are abelian extensions. We note that many cases satisfy the condition in Lemma~\ref{lem-2n} below.
\end{proof}

%
%


In \cite{AG} Proposition 2.11, it was proved that if $Y$ is a connected quandle
and $X=\varphi(Y)\subset {\rm Inn}(Y)$, then each fiber has the same
cardinality, and if $S$ is a set with the same cardinality as a fiber, then
there is a constant cocycle $\beta: X \times X \rightarrow {\rm Sym}(S)$ such
that $Y$ is isomorphic to $X \times_\beta S$. 

\begin{proposition}\label{prop-nonab}
	The quandles $Q(30, 4)$, $Q(36,58)$, and $Q(45, 29)$ are 
	 non-abelian extensions of the quandles $Q(10,1)$, $Q(12,10)$ and $Q(15,7)$,
	 respectively, 
	by constant $2$-cocycles.
\end{proposition}

\begin{proof}
	By calculation we see that the image of the mapping $\varphi$  from
	$Q(30,4)$ (resp. $Q(36,58)$, $Q(45, 29)$) to 
	its inner-automorphism group is isomorphic to $Q(10, 1)$
	(resp. $Q(12,10)$, $Q(15,7)$).  The claim
	follows from  \cite{AG}, Proposition 2.11. 
\end{proof}

We noticed that some non-cohomologous cocycles give isomorphic extensions,
such as  $Q(36,57)$ over $Q(12,10)$ as in the proof of Proposition~\ref{prop-30-4}. 
We also had the following observation from computer calculations.

\begin{remark}\label{rem-noneq}
	{\rm
	Let $X=Q(15, 2)$, which has cohomology group 
	$H^2_Q(X, \Z_2) \cong \Z_2 \times \Z_2$. 
	Hence there are three $2$-cocycles that are
	non-trivial and pairwise non-cohomologous. There are, however, only two
	non-isomorphic abelian extensions   of $X$ by $\Z_2$, $Q(30,1)$ and $Q(30,5)$.  Then
	calculations show that two non-cohomologous cocycles define the extension
	$Q(30,5)$.  Similar examples are found for some 12 element quandles, see
	below. 
	}
\end{remark}

\begin{lemma}\label{lem-ABC}
	For abelian groups $B$ and $C$ and a quandle $X$, 
	let 
	\[
	\phi_B: X \times X \rightarrow B
	\text{ and }
	\phi_C: X \times X \rightarrow C
	\]
	be $2$-cocycles with abelian extensions $E(X,B,\phi_B ) $ and $E(X,C,\phi_C)
	$, respectively.  Then for $A=B \times C$, $\phi=(\phi_B, \phi_C) : X
	\times X \rightarrow A$ is a $2$-cocycle with abelian extension $E(X,A,
	\phi)$, and $E(X,A, \phi)$ is an abelian extension of $E(X,B,\phi_B ) $ and
	$E(X,C,\phi_C) $. 
\end{lemma}

\begin{proof}
	Define  $\phi'_C : E(X, B,  \phi_B) \times E(X, B,  \phi_B)   \rightarrow C$ 
	by \[
	\phi'_C ( \ (x_1, b_1), (x_2, b_2 ) \ )
	=\phi_C (x_1, x_2).
	\]
	Then $\phi'_C$ is a $2$-cocycle of $E(X, B, \phi_B)$ with coefficient $C$.

	Define 
	$f: E(X, A ,  \phi ) \rightarrow E(X, B, \phi_B) \times  C$
	by 
	$f( \ ( x, (b, c) \ ) = ( \ (x, b) , \ c \ )$, which is clearly bijective.
	Then one computes 
	\begin{eqnarray*}
		\lefteqn{ f(\ (x_1, (b_1, c_1)  )*(x_2,  (b_2, c_2) )  \ ) }\\
		&=& 
		f(\ (x_1*x_2, \ (b_1, c_1)  +  \phi (x_1, x_2) ) \ ) \\
		&=& 
		f (\  (x_1*x_2, \ (  b_1 + \phi_B(x_1, x_2) , c_1 +  \phi_C (x_1, x_2) )  \ ) \\
		&=&
		(\  (x_1*x_2 , \  b_1 +  \phi_B (x_1, x_2) ) , \  c_1 +  \phi_C (x_1, x_2)  \ ),
	\end{eqnarray*}
	and 
	\begin{eqnarray*}
		\lefteqn{  f(\ (x_1, (b_1, c_1)  ) \ ) * f(\ (x_2, (b_2, c_2)  \ ) }\\
		&=&
		(\  (x_1, b_1)  , \  c_1\ ) *  (\ (x_2,  b_2 ), \ c_2 ) \ )   \\
		&=&
		(\  (x_1*x_2 , b_1  +  \phi_B(x_1,x_2) ),  \ c_1 + \phi'_C (  \ (x_1, c_1), (x_2, c_2) \ )   \ )  \\
		&=& 
		(\  (x_1*x_2 , b_1  +  \phi_B(x_1,x_2) ),  \ c_1 + \phi_C (  x_1,  x_2 )   \ ),
	\end{eqnarray*}
	as desired.
\end{proof}

Similarly, we obtain the following. 

\begin{lemma}\label{lem-BXC}
    Let $B$ and $C$ be abelian groups and $A=B\times C$, 
    $X$ be a quandle, 
    and $\phi : X \times X \rightarrow A$ be a $2$-cocycle with abelian extension
    $E(X,A,\phi)$. 
    Further, let $p_B$ and $p_C$ be the projections from $A$ onto $B$ and $C$ respectively. 
    Then
    $p_B \phi : X \times  X \rightarrow B$  is a $2$-cocyle giving 
    abelian extension $E(X,B,p_B \phi)$, and
    $E(X,A,\phi) $ is isomorphic to $E(E(X,B,p_B \phi),C, \phi')$,  where $\phi'((x_1,b_1),(x_2,b_2)) = p_C \phi(x_1,x_2)$
    for $(x_1,b_1),(x_2,b_2) \in X \times B$.
\end{lemma}

Lemma~\ref{lem-BXC} is generalized as follows. 

\begin{proposition}\label{prop-exact}
    Let $X$ be a finite quandle, and $0 \rightarrow C \stackrel{\iota}{\longrightarrow} A \stackrel{p_B}{\longrightarrow} B \rightarrow 0$ be an exact sequence of finite abelian groups.
   Let $\phi: X \times X \rightarrow A$ be a quandle $2$-cocycle.
    Then $E(X, A, \phi)$ is an abelian extension of $E(X, B, p_B \phi)$ with coefficient group $C$. 
\end{proposition}

\begin{proof} Let $s: B \rightarrow
A$ be a section of the map $p_B$, that is, $p_Bs= {\rm id}_B$.  Then $$p_B(s(b_1+b_2)-s(b_1)-s(b_2)) = 0.
$$ Thus
$s(b_1+b_2)-s(b_1)-s(b_2)$ lies in the kernel of $p_B$ so we can write
  $$s(b_1+b_2)-s(b_1)-s(b_2) = \iota (c) $$ 
  for some $c \in C$. Let $\eta: B \times B \rightarrow C$ be given by 
 $  \eta (b_1, b_2) = c$. 
  Then $p_B(a-sp_B(a) ) = 0$ and hence we can write $a-sp_B(a) = \iota(p_C(a))$ where
 $p_C: A \rightarrow C$. This yields
 $$\iota p_C (a) + s p_B (a) = a$$
 for all $a \in A$. 

Define $\phi' : E(X, B, p_B \phi) \times E(X, B, p_B \phi) \rightarrow C$ by 
\[
    \phi' (\, (x_1, b_1) , (x_2, b_2) \, )= p_C \phi(x_1, x_2) - \eta (b_1, p_B \phi(x_1, x_2) )
\]
for $ (x_i, b_i)\in  E(X, B, p_B\phi)= X \times B$, $i=1,2$. 
To show that $\phi'$ is a 2-cocycle it suffices to show that $E(E(X,B,p_B \phi),C,\phi')$ is a quandle.
For this it suffices to show that the mapping
 $$f: E( E(X, B, p_B\phi),C,\phi') \rightarrow E(X,A,\phi)$$ defined by
$f( \, ((x, b) , c )\, ) = (x, s(b) + \iota (c) )$ is a bijection and preserves the product.
To show that $f$ is a bijection,
 since the domain and codomain of $f$ have the same cardinality,
  it suffices to show that $f$ is a surjection. Given $(x,a) \in X \times A$ we see that 
$$f((x,p_B(a)),p_C(a) ) = (x,sp_B(a) + \iota p_C(a) ) = (x,a).$$
Finally to show that
$f$ preservers the product we compute:
\begin{eqnarray*}
\lefteqn{ f(\ ( (x_1, b_1), c_1)  * ( (x_2,  b_2), c_2)   \ ) }\\
&=& 
f(\, ( (x_1, b_1)* (x_2, b_2) , \, c_1 + \phi'( \, (x_1, b_1) , (x_2, b_2)\,  ) \, ) \, )\\
&=&
f(\, ( (x_1* x_2,  b_1+ p_B \phi (x_1, x_2) ) , \, c_1 + p_C \phi(x_1, x_2) - \eta (b_1, \phi(x_1, x_2) ) ) \, ) \\
&=&
(\,  x_1 * x_2, \, s( b_1+ p_B \phi (x_1, x_2) ) + \iota( c_1 + p_C \phi(x_1, x_2) - \eta (b_1, \phi(x_1, x_2) )\,  )\, ) \\
&=&
(\,  x_1 * x_2, \, s( b_1) + s p_B \phi (x_1, x_2) + \iota \eta (b_1, p_B \phi (x_1, x_2) ) \\
& & 
\mbox{ \hspace{35pt} } + \iota( c_1) + \iota p_C \phi(x_1, x_2) - \iota \eta (b_1, p_B \phi(x_1, x_2) )\,  ) \\
&=&
(\,  x_1 * x_2, \, s( b_1)   
 + \iota( c_1) +  \phi(x_1, x_2) \, ), 
\end{eqnarray*}
and 
\begin{eqnarray*}
\lefteqn{  f(\ ( (x_1, b_1), c_1)   \ ) * f(\ ( (x_2, b_2), c_2)  \ ) }\\
&=&
(x_1, \, s(b_1)+ \iota( c_1 ) ) *   (x_2,\,  s(  b_2 ) + \iota (  c_2  ) ) \\
&=&
( x_1 * x_2, \, s( b_1)  
 + \iota( c_1) +  \phi(x_1, x_2) ),
\end{eqnarray*}
as desired.
\end{proof}

If we  suppress the 2-cocycle in the notation $E(X,A,\phi)$ and write merely
$E(X,A)$ then the above 
Lemma~\ref{lem-ABC} and Proposition~\ref{prop-exact} 
may be stated more simply.

\begin{corollary} 
{\rm (i)} If $E(X, B)$ and $E(X, C)$ are abelian extensions, then so is $E(X, B \times C)$,
and 
$$E(X, B \times C)=E(E(X, B), C). $$

{\rm (ii)}   If $E(X,A)$ is a finite abelian extension of a quandle $X$ and $C$ is a
    subgroup of the  finite abelian group  $A$ then 
    \[
        E(X,A) = E(E(X,A/C),C).
    \]
\end{corollary}

We note that if $E(X, A)$ is connected, then $E(X, A/C)$ is connected since the
epimorphic image of a connected quandle is connected.

We examine some connected abelian extensions of Rig quandles of order up to 12.
In the following, we use the notation $E  \stackrel{n}{\longrightarrow} X$ if
$E=E(X, \Z_n. \phi)$ for some $2$-cocycle $\phi$ such that $E$ is connected.
$E_2 \stackrel{m}{\Longrightarrow} E_1 \stackrel{d}{\Longrightarrow} X$ if
there is a short exact sequence $0 \rightarrow \Z_m \rightarrow \Z_n
\rightarrow \Z_d\rightarrow 0$ such that $\Z_n \subset H^2_Q(X, \Z_n)$ and
$E_1, E_2$ are corresponding extensions as in Proposition~\ref{prop-exact}.
In this case $E_2 \stackrel{n}{\longrightarrow} X$ where $n=md$.  The notation
$ \emptyset \stackrel{1}{\longrightarrow} X$ indicates that $H^2_Q(X, A)=0$ for
any coefficient group $A$, and hence there is no non-trivial abelian extension.
It is noted to the left when all quandles in question are keis.
\begin{eqnarray*}
	\emptyset \stackrel{1}{\longrightarrow} Q(8,1) \stackrel{2}{\longrightarrow} Q(4,1) & & \\
	{\rm (Kei)} \quad \emptyset \stackrel{1}{\longrightarrow} Q(24,1) \stackrel{2}{\longrightarrow} Q(12,1) \stackrel{2}{\longrightarrow}Q(6,1) & & \\
	\emptyset \stackrel{1}{\longrightarrow} Q(24,2) \stackrel{2}{\Longrightarrow} Q(12,2) \stackrel{2}{\Longrightarrow}Q(6,2)  & & \\
	{\rm (Kei)} \quad  \emptyset \stackrel{1}{\longrightarrow} Q(27, 1) \stackrel{3}{\longrightarrow} Q(9,2)=Q(3,1)\times Q(3,1) & & \\
	\emptyset \stackrel{1}{\longrightarrow} Q(27, 6) \stackrel{3}{\longrightarrow} Q(9,3)=\Z_3[t]/(t^2+1)& & \\
	\emptyset \stackrel{1}{\longrightarrow} Q(27, 14) \stackrel{3}{\longrightarrow} Q(9,6)=\Z_3[t]/(t^2+2t+1)
	& & \\
	\emptyset \stackrel{1}{\longrightarrow}  Q(24,8) = Q(3,1)
	\times Q(8,1) \stackrel{2}{\longrightarrow} Q(12,4)=Q(3,1)
	\times Q(4,1) & & 
\end{eqnarray*}
 
In the following, we list abelian extensions of Rig quandles that contain
quandles of order higher than $35$.  The notation $Q(n, -)$ indicates that it
is a quandle of order $n>35$ and is not a Rig quandle.  
The notation  $?
\longrightarrow Q(n, - )  $ indicates that we do not know if non-trivial
abelian extension exists for the quandle $Q(n, -)$ in question.  
Except for the quandle $Q(120,-)$ in the third line, we have explicit quandle operation tables for
the quandles appearing in the list and hence we can prove by computer that such quandles are connected.
   
\begin{eqnarray*}
 ? \longrightarrow Q(120, - )  \stackrel{6}{\longrightarrow} Q(20,3) \stackrel{2}{\longrightarrow} Q(10,1) & &  \\
 ? \longrightarrow Q(120, - ) \stackrel{5}{\Longrightarrow} Q(24,7)  \stackrel{2}{\Longrightarrow}  Q(12,3)   & & \\
 ? \longrightarrow Q(120, - ) \stackrel{2}{\Longrightarrow} Q(60, -)  \stackrel{5}{\Longrightarrow}  Q(12,3)   & & \\
  ? \longrightarrow Q(48, - ) \stackrel{2}{\Longrightarrow}  Q(24, 4) \stackrel{2}{\Longrightarrow} Q(12,5)  & & \\
    ? \longrightarrow Q(48, - )  \stackrel{2}{\Longrightarrow}  Q(24, 3) \stackrel{2}{\Longrightarrow} Q(12,6)   & & 
  \end{eqnarray*}
  It is interesting to remark that 
all   quandles appearing in the first and the last lines are keis.

These observations raise the following questions.

$\bullet$ What is a condition on cocycles  for  abelian, or non-abelian  extensions
to be connected?

In \cite{AG}, a condition for an extension to be connected was given in terms
of elements of the inner automorphism group.

$\bullet$ Is there an infinite sequence of abelian extensions of connected
quandles $\cdots \rightarrow Q_n \rightarrow \cdots \rightarrow Q_1$?

We note that sequences of abelian extensions of connected quandles terminate as much as we were able to 
compute. 

$\bullet$ Is any abelian extension of a finite kei  a kei?

In relation to this question, below we observe a condition of $2$-cocycles that
give extensions that are keis.

\begin{lemma}\label{lem-kei}
    Let $X$ be a kei, $\phi$ be a $2$-cocycle with coefficient group $A$,
    and $E$ be the abelian extension of $X$ with respect to $\phi$.
    Then $E$ is a kei if and only if $\phi (x,y) + \phi(x*y, y)= 0 \in A$ for any $x, y \in X$,
    in additive notation.
\end{lemma}

\begin{proof}
	One computes, for any $x, y \in X$ and $a, b \in A$,  
	\begin{eqnarray*}
		\lefteqn{ [ \ (x, a)*(y, b) \ ] *(y, b) }\\
		&=& (x*y, a +  \phi (x, y) ) * (y, b) = (\ (x*y)*y, a+  \phi (x, y) +\phi(x*y, y) ).
	\end{eqnarray*}
	For any $x, y \in X$ and $a, b \in A$, the right-hand side is equal to $(x, a)$ if and only if 
	$\phi (x,y) +\phi(x*y, y)=0 $ for any $x, y \in X$. 
\end{proof}

\begin{remark}
	{\rm 
	Let $X=Q(12, 7)$. Then $H^2_Q(X, \Z_4) \cong \Z_2 \times \Z_4$.  By
	computer calculation, there is a particular generating $2$-cocycle of the
	$\Z_2$-factor, $Q(24, 15) \stackrel{2}{\longrightarrow} Q(12,7)$.  Notice
	that $H^2_Q( Q(24, 15), \Z_4) \cong \Z_4$.  We also note that there are
	epimorphisms $Q(24, 14) {\rightarrow} Q(12,7) $ and $Q(24, 18) {\rightarrow
	} Q(12,7) $, where $H^2_Q( Q(24, 14), \Z_4) \cong \Z_4$ and $H^2_Q( Q(24,
	14), \Z_2) \cong \Z_2 \times \Z_2$.  Hence there is a quandle of order $48$
	corresponding to the $\Z_4$-factor of $H^2_Q(X, \Z_4)$, that has epimorphic
	image $Q(24, 14)$ or $Q(24, 18)$.
	}

%
%
%
\end{remark}

\begin{remark}
{\rm 
Let $X=Q(12, 8)$. Then $H^2_Q(X, \Z_2) \cong (\Z_2)^3$. 
There are three epimorphisms from Rig quandles 
of order less than 36: 
\begin{align*}
Ê Q(24, 5)\to X, &&Q(24, 16)\to X,&&Q(24, 17)\to X 
\end{align*}
and their cohomology groups with $A=\Z_2$ 
are $(\Z_2)^3$, $(\Z_2)^2$,  and $(\Z_2)^2$, respectively.
We note that there are 7 cocycles that are not cohomologous each other, 
yet there are only 3 extensions as in Remark~\ref{rem-noneq}.
}
\end{remark}

%
%
%


\begin{remark}
{\rm 
    Let $X=Q(12, 9)$. Then $H^2_Q(X, \Z_4) \cong \Z_4 \times \Z_4$. 
There are two extensions in Rig quandles of order less than 36:
\[
    Q(24, 6)\rightarrow Q(12,9),\quad Q(24, 19) \rightarrow Q(12,9)
\]
and 
with $A=\Z_4$ their cohomology groups are $\Z_2 \times \Z_2 \times \Z_4$ and $\Z_2 \times \Z_4$, 
respectively. There are 3 cocycles that give order 2 extensions, yet there are two extensions
as in Remark~\ref{rem-noneq}.
%
%
%
}
\end{remark}

\begin{remark}
{\rm 
Let $X=Q(12, 10)$. Then $H^2_Q(X, \Z_6) \cong   \Z_6$.  There is one extension
among  Rig quandles of order less than 36, $Q(24, 20)  \rightarrow Q(12,10) $ and we have
$H^2_Q(Q(24, 20), \Z_3) \cong   \Z_3 \times \Z_3$.  One of the order 3 cocycle
corresponds to an extension of $X$ of order 6.
}
\end{remark}

\section{Finding extensions of higher order}\label{sec-higher}

We further investigated extensions among non-faithful quandles over Rig
quandles.  Extensions of some of the Rig quandles of order greater than 12 and
less than 28 can be found in 
\url{http://github.com/vendramin/rig/wiki}.
The computations of cocycles become difficult for quandles of order 28.  Thus
we take an approach of constructing non-faithful connected quandles and
identify extensions as follows. 

We considered Rig quandles of order less than 36. %
To find
extensions of Rig quandles, we made a list ${\cal N \cal F}$  of 315
non-faithful connected generalized Alexander quandles with respect to pairs
$(G, f)$ for non-abelian groups $G$ and $f \in {\rm Aut}(G)$ (see
Section~\ref{sec-prelim}). We considered all
groups or order $n$, $36 \leq n <128$, and for $n=128$, only the first $172$
groups in \GAP~ Small Groups library (the library contains all the $2328$
groups of size $128$).  All possible automorphisms $f \in {\rm Aut}(G)$ were
considered up to conjugacy.  For example, there are 39 non-abelian groups of
order 108 which give 74 connected non-faithful quandles of order 108. 

Lemma 2.3 and Proposition 2.11  in \cite{AG} were used to determine abelian
extensions and non-abelian extensions by constant cocycles among quandles in
${\cal N \cal F}$ over Rig quandles of order less than 36.  %
Specifically, quotient quandles are
computed,  {\it dynamical } cocycles (\cite{AG}, Lemma 2.3) are computed,
whether the cocycles are constant is determined, and whether the extensions are
abelian is determined.

We note that most examples computed for abelian extensions are 2-fold
epimorphisms, and observe the following.

\begin{lemma}\label{lem-2n}
    Let $Y$ be a finite connected 
    quandle of  even order $2n$, 
    and assume that $\varphi(Y)=X \subset {\rm Inn}(Y)$ with  $|X|=n$. 
    Then $Y$ is isomorphic to 
    an abelian extension of $X$ by $\Z_2$.
\end{lemma}

\begin{proof}
	As in  Proposition~\ref{prop-nonab}, it follows from Proposition 2.11 of
	\cite{AG} that $Y$ is  isomorphic to  an extension  $X \times_\beta S$ by a
	constant cocycle $\beta$, where a set $S$ consists of two elements.  Let
	$S=\{0, 1\}$, and we identify $S$ with $\Z_2$. Then ${\rm Sym}(S)$
	consists of two elements, the identity and the transposition of $0$ and
	$1$.  We define $\phi: X \times X \rightarrow \Z_2$ by $\phi(x, y)=0$ if
	$\beta_{x,y} = {\rm id}$ and $\phi(x, y)=1$ if $\beta_{x,y}$ is the
	transposition.  Then $\beta_{x,y}(t)=t + \phi(x,y)$ for $t \in \Z_2$ and
	$\phi$ is a $2$-cocycle.
\end{proof}

\begin{remark}
{\rm 
	Among Rig quandles of order less than 36, the following %
	have 2-fold extensions among quandles in ${\cal N \cal F}$.
	\begin{eqnarray*}
		Q(18, i) , & & i=1,3, 4, 5, 6, 7, 8.\\
		Q(24, i), & & i=6, 10, 11, 13, 18, 22,23. \\
		Q(28, i) , & & i=1, 2, 3, 4, 5, 6, 7, 8, 9. 
	\end{eqnarray*}
	Other than these, we found that $Q(12,3) $ has a 5-fold abelian extension, and 
	$Q(15,2)$ has a 4-fold non-abelian extension in ${\cal N \cal F}$.
    We remark that the 5-fold extension of $Q(12,3)$ was predicted by
    Lemma~\ref{lem-BXC}, see the list in Section~\ref{sec-ext} for $Q(12,3)$.
    Thus this specific extension is found in ${\cal N \cal F}$.
	}
\end{remark}

We observe the following generalization of Lemma~\ref{lem-2n}. 
Let $Y$ be a finite connected 
	quandle, and let  $\varphi(Y)=X \subset {\rm Inn}(Y)$ with  $|X|=n$.  
	It follows again from Proposition 2.11 of \cite{AG} that $Y$ is isomorphic
    to an extension   $ X \times_\beta S$ 
    by a constant cocycle $\beta$.

\begin{lemma}
	Let $X$ and $Y$ be as above. 
	If
	$|Y|=k n$
	where $k$ is a prime power, and the subgroup $H_\beta$ of ${\rm Sym}(S)$ generated by $\{
	\, \beta_{x,y} \, | \, x,y \in X \, \}$ is cyclic of order $k$, then $Y$ is isomorphic to
	an abelian extension of $X$.
\end{lemma}

\begin{proof}
Since $|Y|=kn$, we have $|S|=k$. 
Since $H_\beta$ is cyclic of order $k$
and $k$ is a prime power, 
it is generated by a $k$-cycle $\sigma$. 
 We can identify  $S$ with $\Z_k$ in such a way that $\sigma=(1, 2, \ldots, k)$.
 Then $\sigma (t) = 1+t$ for any $t \in \Z_k$.
Hence for any 
    $x, y \in X$, $\beta_{x,y}=\sigma^i$ for some $i \in \Z_k$,  
   so that 
     $\beta_{x,y}(t)=t+\phi(x, y)$ with $\phi(x, y)=i$. 
%
\end{proof}

%

\begin{remark}
{\rm
Although homology groups of Rig quandles have been computed in \cite{rig},  as
mentioned earlier, explicit $2$-cocycles have not been computed for Rig
quandles of order greater than 23. The above computations of extensions give
rise to explicit $2$-cocycles, and also may be used for computations of cocycle
invariants as in Section~\ref{sec-recover}.  Furthermore, the computations
identify the pairs $(G, f)$ of generalized Alexander quandles that are abelian
extensions of Rig quandles.
}
\end{remark}

\section{Problems, questions and conjectures}

For convenience of the reader, we collect here questions, problems and
conjectures discussed all over the text.

\begin{problem}
    Compute explicit $2$-cocycles and extensions of Rig quandles of order $\geq 24$.
\end{problem}

In Remark \ref{rem-endmono} we made the following conjecture.

\begin{conjecture}
  Let $X$ be one of the following quandles:
    \begin{align*}
        & Q(12,1), && Q(20,3), && Q(24,3), && Q(24,4), && Q(24,5),\\
        & Q(24,6), && Q(24,14), && Q(24,16), && Q(24,17), && Q(30,1),\\
        & Q(30,16), && Q(32,5), && Q(32,6), && Q(32,7), && Q(32,8).
    \end{align*}
    Then every knot $K$ is end monochromatic with $X$.
\end{conjecture}

In Examples \ref{exa:Q(12,5)}, \ref{exa:Q(12,6)} and \ref{exa:Q(12,10)} we made
the following conjectures.

\begin{conjecture}
    Let $X=Q(12,5)$ and $\phi$ be the $2$-cocycle choosen in Example
    \ref{exa:Q(12,5)}. Then for each knot $K$ the cocycle invariant $\Phi_\phi$
    is of the form $\Phi_\phi(K)=a+bu^2$, where $a,b\in\Z$.
\end{conjecture}

\begin{conjecture}
    Let $X=Q(12,6)$ and $\phi$ be the $2$-cocycle choosen in Example
    \ref{exa:Q(12,6)}.  Then for each knot $K$ the cocycle invariant
    $\Phi_\phi$ is of the form $\Phi_\phi(K)=a+bu^2$, where $a,b\in\Z$.
\end{conjecture}

\begin{conjecture}
    Let $X=Q(12,10)$ and $\phi$ be the $2$-cocycle choosen in Example
    \ref{exa:Q(12,10)}. For each knot $K$ write
    $\Phi_\phi(K)=a+bu+cu^2+du^3+eu^4+fu^5$, where $a,b,c,d,e,f\in\Z$. Then
    $b=f=0$ for all $K$.
\end{conjecture}

In Section 6 
we posed the following questions.

\begin{question}
What is a condition on cocycles  for  abelian, or non-abelian  extensions to be connected?
\end{question}

\begin{question}
Is there an infinite sequence of abelian extensions of connected quandles
$\cdots \rightarrow Q_n \rightarrow \cdots \rightarrow Q_1$ ?
\end{question}

\begin{question}
Is any abelian extension of a finite kei  a kei?
\end{question}

\appendix

\section*{Appendix: Cocycle invariants for $Q(6,2)$}

In this appendix we list the cocycle invariant $\Phi_\phi(K)$ for the quandle
$X=Q(6,2)$ and the 2-cocycle over $\Z_4$ discussed in Example~\ref{ex-62}.  The
list is for all knots in \cite{KI} that have braid index $4$ or less, and 12
crossings or less.  Knots with only trivial colorings (the invariant value $6$)
are not listed.  These values are computed using the formula described in
Example~\ref{ex-62} and  programs similar to those  in \cite{CESY}.  Note that
if the 2-cocycle invariant below has the form $a+bu+cu^2+du^3$ where $b \neq d$
then  by Lemma 4.3 each of  the corresponding knots $K$ satisfes $K \neq
rm(K)$. 

\newpage
\begin{longtable}{|c|c|}
    \caption{Some cocycle invariants for the quandle $Q(6,2)$ of the form
    $a+bu$ for some $a,b\in\Z$.}\\
        \hline
        cocycle invariant & knot \\
        \hline
        $54$ & $10_{62}, 10_{65}, 10_{140}, 10_{143}, 10_{165}$  \\
        &  $11a_{108}, 11a_{109}, 11a_{139}, 11a_{157}$\\
        &  $11n_{85}, 11n_{106}, 11n_{118}, 11n_{119}$ \\ 
        &  $12a_{0290}, 12a_{0375}, 12a_{0390}, 12a_{0571}$\\
        &  $12a_{0668}, 12a_{0672}, 12a_{0941}, 12a_{0949}$\\
        &  $12a_{1184}, 12a_{1191},  12a_{1207}, 12a_{1215}$\\
        &  $12n_{0425}, 12n_{0426},12n_{0533}, 12n_{0807}$\\
        &  $12n_{0811}, 12n_{0812}, 12n_{0831}, 12n_{0868}$\\
        \hline
        $198$ & $10_{99}$\\
        \hline
        $6+24 u$ & $3_{1}, 8_{11}, 9_{4}, 9_{10}, 9_{17}, 9_{34}$ \\ 
        & $10_{5}, 10_{9}, 10_{40}, 10_{103}, 10_{106}$\\
        & $10_{136}, 10_{146},  10_{158}, 10_{159}, 10_{163}$ \\ 
        & $11a_{73}, 11a_{99}, 11a_{146}, 11a_{171}, 11a_{175}$\\
        & $11a_{176}, 11a_{184}, 11a_{196}, 11a_{216}, 11a_{239}$\\
        & $11a_{248}, 11a_{306},11a_{346}, 11a_{353},  11n_{13}$\\
        & $11n_{14}, 11n_{86}, 11n_{98}, 11n_{109}$\\
        & $11n_{125}, 11n_{137}, 11n_{138},  11n_{158}$  \\
        & $12a_{0234}, 12a_{0346}, 12a_{0409}, 12a_{0411}$\\
        & $12a_{0422}, 12a_{0509},12a_{0519}, 12a_{0523}$\\
        & $12a_{0567}, 12a_{0588}, 12a_{0617}, 12a_{0626}, 12a_{0718}$\\
        & $12a_{0723}, 12a_{0878}, 12a_{0894}, 12a_{0904}, 12a_{0907}$\\
        & $12a_{0916}, 12a_{0923},  12a_{0944}$ \\
        & $12a_{0986}, 12a_{1002}, 12a_{1025}, 12a_{1029}$\\
        & $12a_{1060}, 12a_{1079}, 12a_{1115}, 12a_{1120}$\\
        & $12a_{1136}, 12a_{1170}, 12a_{1177}, 12a_{1180}$\\
        & $12a_{1197}, 12a_{1201}, 12a_{1214}, 12a_{1226}$\\
        & $12a_{1247}, 12a_{1248}, 12a_{1262}, 12a_{1270}$\\
        & $12a_{1272}, 12a_{1276}, 12n_{0147}, 12n_{0329}$\\
        & $12n_{0369}, 12n_{0377}, 12n_{0409}, 12n_{0413}$\\
        & $12n_{0419}, 12n_{0439}, 12n_{0493}, 12n_{0502}$\\
        & $12n_{0543}, 12n_{0597}, 12n_{0653}, 12n_{0655}$\\
        & $12n_{0657}, 12n_{0660}, 12n_{0667}, 12n_{0668}$\\
        & $12n_{0752}, 12n_{0767}, 12n_{0782}, 12n_{0803}$\\
        & $12n_{0825}, 12n_{0866}, 12n_{0284}$\\
        \hline
        $54+72 u$ & $12n_{0546}$\\
        \hline 
        $150+24 u$ & $11n_{126}, 12n_{0440}$ \\
        \hline
\end{longtable}

\newpage
\begin{longtable}{|c|c|}
    \caption{Some cocycle invariants for the quandle $Q(6,2)$ of the form $a+bu+cu^2$
    for some $a,b,c\in\Z$ with $c\ne0$.}\\
        \hline
        cocycle invariant & knot \\
        \hline
        $6+48 u^2$ & $9_{40}, 10_{61}, 10_{64}, 10_{66}$\\
        & $10_{139}, 10_{141}, 10_{142}, 10_{144}, 10_{164}$\\
        & $11a_{106}, 11a_{194}, 11a_{223}, 11a_{232}$\\
        & $11a_{244}, 11a_{338}, 11a_{340}, 11n_{87}$  \\
        & $11n_{104}, 11n_{105}, 11n_{107}, 11n_{145}$\\
        & $11n_{146}, 11n_{173}, 11n_{183}, 11n_{184},  11n_{185}$ \\
        & $12a_{0428}, 12a_{0670}, 12a_{0737}, 12a_{0739}, 12a_{0855}$\\
        & $12a_{0864}, 12a_{0970}, 12a_{1111}, 12a_{1147}, 12a_{1212}$\\
        & $12a_{1219}, 12a_{1221}, 12n_{0483}, 12n_{0484}, 12n_{0536}$\\
        & $12n_{0627}, 12n_{0779}$ \\
        \hline
        $6+48 u+96 u^2$ & $12a_{0701}, 12a_{0987}$\\
        \hline
        $30+24 u^2$ & $8_{5}, 8_{10}, 8_{15}, 8_{19}, 8_{20}, 8_{21}, 9_{16}, 9_{24}, 9_{28}$ \\
         & $10_{76}, 10_{77},  10_{82}, 10_{84}, 10_{85}, 10_{87}$ \\
         & $11a_{71}, 11a_{72}, 11a_{245}, 11a_{261}$\\
         & $11a_{264}, 11a_{305}, 11a_{351}$\\
         & $11n_{38},  11n_{121}$  \\
         & $12a_{0577}, 12a_{0578}, 12a_{0852}$\\
         & $12a_{0861}, 12a_{0930}, 12a_{0979}$ \\
         & $12a_{0981}, 12a_{0982}, 12a_{0999}$\\
         & $12a_{1000}, 12a_{1059}, 12a_{1061}$  \\
         & $12a_{1100}, 12a_{1187}, 12a_{1252}$\\
         & $12a_{1253}, 12a_{1261}, 12a_{1284}$  \\
         & $12a_{1285}, 12n_{0084}, 12n_{0106}$ \\
         & $12n_{0107}, 12n_{0290}, 12n_{0291}$  \\
         & $12n_{0572}, 12n_{0573}, 12n_{0575}$ \\
         & $12n_{0576}, 12n_{0577}, 12n_{0578}$ \\
         & $12n_{0638}, 12n_{0674}, 12n_{0675}$ \\
         & $12n_{0700}, 12n_{0753}, 12n_{0833}$  \\
         & $12n_{0845}, 12n_{0850}$ \\
        \hline
        $30+168 u^2$ & $12n_{0604}$\\
        \hline
        $54+144 u^2$ & $12n_{0508}$\\
        \hline
        $54+48 u+48 u^2$ & $12a_{0742}, 12n_{0380}$\\
        \hline
        $78+48 u+24 u^2$ & $12a_{0574}, 12n_{0571}, 12n_{0574}$ \\
        \hline
        $102+96 u^2$ & $12n_{0518}$\\
        \hline
        $126+72 u^2$ & $12a_{0647}, 12n_{0605}$\\
        \hline
        $150+48 u^2$ & $12a_{1288}, 12n_{0888}$ \\
        \hline
\end{longtable}

\newpage
\begin{longtable}{|c|c|}
    \caption{Some cocycle invariants for the quandle $Q(6,2)$ of the form $a+bu+cu^2+du^3$
    for some $a,b,c,d\in\Z$ with $d\ne0$.}\\
        \hline
        cocycle invariant & knot \\
        \hline
        $6+24 u^3$ & $6_{1}, 7_{4}, 7_{7}, 9_{1}, 9_{6}, 9_{11}, 9_{23}, 9_{29}, 9_{38}$ \\
        & $10_{14}, 10_{19}, 10_{21}, 10_{32}$\\
        & $10_{108}, 10_{112}, 10_{113}, 10_{114}$  \\
        & $10_{122}, 10_{145}, 10_{147},  10_{160}$  \\
        & $11a_{179}, 11a_{203}, 11a_{236}, 11a_{274}$\\
        & $11a_{286}, 11a_{300}, 11a_{318}, 11a_{335}$\\
        & $11a_{355}, 11a_{365}, 11n_{65}, 11n_{66}, 11n_{92}$ \\
        & $11n_{94}, 11n_{95}, 11n_{99}, 11n_{122}, 11n_{136}$\\
        & $11n_{143}, 11n_{148}, 11n_{149}, 11n_{153}, 11n_{176}$\\
        & $11n_{182}, 12a_{0236}, 12a_{0321}, 12a_{0496}$\\
        & $12a_{0580}, 12a_{0762}, 12a_{0805}$  \\
        & $12a_{0806}, 12a_{0807},  12a_{0809}$\\
        & $12a_{0876}, 12a_{0909}, 12a_{0952}$  \\
        & $12a_{0972}, 12a_{1036}, 12a_{1091}$\\
        & $12a_{1101}, 12a_{1129}, 12a_{1157}$ \\
        & $12a_{1196}, 12a_{1200}, 12a_{1210}$\\
        & $12a_{1216}, 12a_{1224}, 12a_{1237}$ \\
        & $12a_{1239}, 12a_{1255}, 12n_{0330}$\\
        & $12n_{0368},  12n_{0375}, 12n_{0412}$  \\
        & $12n_{0438}, 12n_{0441}, 12n_{0443}$\\
        & $12n_{0464}, 12n_{0500}, 12n_{0603}$  \\
        & $12n_{0640}, 12n_{0641},12n_{0717}$\\
        & $12n_{0738}, 12n_{0740}, 12n_{0750}$  \\
        & $12n_{0751}, 12n_{0754},12n_{0769}$\\
        & $12n_{0770}, 12n_{0781}, 12n_{0791}$  \\
        & $12n_{0823}, 12n_{0832}, 12n_{0836}$\\
        & $12n_{0865},  12n_{0874}, 12n_{0875}$  \\
        & $12n_{0882}$\\
        \hline
        $6+48 u^2+72 u^3$ & $9_{48}, 11a_{293}, 12a_{0895}$\\
        \hline
        $6+144 u^2+24 u^3$ & $10_{98}$\\
        \hline
        $6+24 u+72 u^3$ & $12n_{0666}$\\
        \hline 
        $6+24 u+48 u^2+48 u^3$ & $11n_{164}, 12n_{0402}$\\
        \hline 
        $6+24 u+96 u^2+24 u^3$ & $11n_{167}$\\
        \hline
        $6+48 u+48 u^3$ & $8_{18}, 12a_{1260}, 12n_{0403}$\\
        \hline 
        $6+48 u+48 u^2+24 u^3$ & $12n_{0565}$\\
        \hline        
        $6+72 u+24 u^3$ & $9_{47}, 12n_{0549}$\\
        \hline 
        $30+120 u^2+24 u^3$ & $12n_{0737}$\\
        \hline   
        $30+24 u+72 u^2+24 u^3$ & $12a_{0576}, 12n_{0570}$\\
        \hline
        $54+72 u^3$ & $11a_{314}$\\
        \hline
        $54+48 u^2+48 u^3$ & $11a_{332}, 12n_{0386}$\\
        \hline
        $54+24 u+48 u^2+24 u^3$ & $12a_{0297}, 12n_{0379}$\\
        \hline
        $54+48 u+24 u^3$ & $9_{46}, 11a_{291}, 12n_{0567}$\\
        \hline
        $78+24 u^2+48 u^3$ & $12a_{1283}$\\
        \hline
        $78+24 u+24 u^2+24 u^3$ & $12n_{0883}$\\
        \hline
        $78+48 u+72 u^2+48 u^3$ & $11a_{44}, 11a_{47}, 11a_{57}$\\
        & $11a_{231}, 11a_{263}, 11n_{71}$\\
        & $11n_{72}, 11n_{73}, 11n_{74}$\\
        & $11n_{75}, 11n_{76}, 11n_{77}$\\
        & $11n_{78}, 11n_{81}$ \\
        & $12a_{0167}, 12a_{0692}, 12a_{0801}$\\        
        \hline
        $102+48 u^3$ & $12n_{0806}$\\
        \hline
        $102+24 u+24 u^3$ & $11a_{277}, 12a_{1225}$\\
        \hline
\end{longtable}

\subsection*{Acknowledgements}
MS was partially supported by
the
National Institutes of Health under Award Number R01GM109459. The content of
this paper is solely the responsibility of the authors and does not necessarily
represent the official views of  NIH.  LV was partially supported by
Conicet, ICTP, UBACyT 20020110300037
and PICT-2014-1376.
We thank Santiago Laplagne for 
the computer were some calculations were performed.

\end{document}